\documentclass[11pt,twoside,a4paper]{article}

\usepackage{amssymb}
\usepackage{amsmath}
\usepackage{amsthm}

\usepackage{mathrsfs}

\usepackage{amsfonts}
\usepackage{amssymb}
\usepackage{amsmath}
\usepackage{amsthm}
\usepackage{bbm}
\usepackage{verbatim}
\usepackage{url}
\usepackage{xcolor}

\allowdisplaybreaks

\newcommand{\avg}[1]{\langle #1 \rangle}

\def\R{{\mathbb R}}

\def\ls{\lesssim}

\newtheorem{thm}{Theorem}[section]
\newtheorem{lem}[thm]{Lemma}
\newtheorem{prop}[thm]{Proposition}
\newtheorem{cor}[thm]{Corollary}
\newtheorem{defn}[thm]{Definition}

\numberwithin{equation}{section}

\begin{document}

\arraycolsep=1pt

\title{\Large\bf   Commutators of multi-parameter flag singular integrals\\ and applications
}
\author{Xuan Thinh Duong, Ji Li, Yumeng Ou, Jill Pipher and Brett D. Wick}

\date{}
\maketitle

\begin{center}
\begin{minipage}{13.5cm}\small

{\noindent  {\bf Abstract:}\  
We introduce the iterated commutator for the Riesz transforms in the multi-parameter flag setting, and prove the upper bound of this commutator with respect to the symbol $b$ in the flag BMO space. Our methods require the  techniques of semigroups, harmonic functions and multi-parameter flag Littlewood--Paley analysis. We also introduce the {\it big} commutator in this multi-parameter flag setting and prove the upper bound with symbol $b$ in the flag little-bmo space by establishing the ``exponential--logarithmic'' bridge between this flag little bmo space and the Muckenhoupt $A_p$ weights with flag structure. As an application, we establish the div-curl lemmas with respect to the appropriate Hardy spaces in the multi-parameter flag setting.
}

\end{minipage}
\end{center}

%
%
\bigskip


\medskip


\section{Introduction and statement of main results}
\setcounter{equation}{0}

The Calder\'{o}n--Zygmund theory of singular integrals has been central to the success and applicability of modern harmonic analysis in the last fifty years.
This theory has had extensive applications to other fields of mathematics such as complex analysis, geometric measure theory and partial differential equations. 
In the setting of Euclidean spaces $\mathbb R^n$, a notable property of 
standard Calder\'{o}n--Zygmund singular integrals, shared with
the Hardy--Littlewood maximal operator, is that these operators commute with
the classical one-parameter family of dilations on $\mathbb{R}^{n}$, $\delta \cdot x=(\delta
x_{1},\ldots,\delta x_{n})$ for $\delta >0$. See for example the monograph \cite{s93}.  

The product Calder\'{o}n--Zygmund theory in harmonic analysis was introduced in the 70s, and studied extensively since then.
The model case is a tensor product of classical singular integral operators; such operators arise in the context of questions about summation of
multiple variable Fourier series.
Early key work in this field includes that of Chang, C. Fefferman, R. Fefferman, Gundy,  Journ\'{e}, Stein \cite{GS, FeS,
F1, F2, F3, CF1,CF2,CF3, J1, P}. Included in these works are the identification of appropriate notions of  product  ${\rm BMO}$ and
product Hardy space $H^{p}\left( \mathbb{R}^{n}\times \mathbb{R}
^{m}\right) $.  

More recently, the theory of (iterated) commutators has been developed in connection with the
Chang--Fefferman BMO space, including paraproducts and multi-parameter div-curl lemmas; see, for example,
\cite{DO,FL,FS,LPPW,LPPW2,LPPW3,LT}.
In contrast with the classical Euclidean setting, the product
Calder\'{o}n--Zygmund singular integrals, and the \emph{strong} maximal function operator, commute with the
multi-parameter dilations on $\mathbb{R}^{n}$, $\delta \cdot x=(\delta
_{1}x_{1},\ldots,\delta _{n}x_{n})$ for $\delta =(\delta _{1},\ldots,\delta
_{n})\in (0,\infty)^{n}$.



A new type of multi-parameter structure, which lies in between one-parameter and tensor product,
was introduced by Muller, Ricci and Stein
in \cite{MRS} and \cite{MRS2}, where they studied the $L^p$ boundedness of Marcinkiewicz multipliers $m(\mathcal L, iT)$
on Heisenberg group, where $\mathcal L$ is the sub-Laplacian and $T$ is the central invariant vector field, with $m$ being a multiplier of Marcinkiewicz-type. They showed that such Marcinkiewicz multipliers can be characterized by
a convolution operator  $f\ast K$ where $K$ is a so-called {\it flag} convolution kernel.
This multi-parameter flag structure is not explicit, but only \emph{implicit} in the sense that one can not formulate it in terms of an explicit dilation $\delta$ acting on $x$.
Later,  the notion of flag kernels (having singularities on appropriate flag varieties) and the properties of the corresponding singular integrals were then extended to the higher step case in Nagel, Ricci and Stein \cite{NRS} on Euclidean space and their applications
on certain quadratic CR submanifolds
of ${\Bbb{C}}^n$. Recently, Nagel, Ricci, Stein and Wainger \cite{NRSW1, NRSW2}  established the
theory of singular integrals with flag kernels in a more general setting of homogeneous groups.
They proved that, on a homogeneous group, singular integral operators with flag kernels are bounded on  $L^p,1<p<\infty,$ and form an algebra.
(See also \cite{G2, G3} for related work.) Associated to this implicit multi-parameter flag structure,
the Hardy space $H^1_{\mathcal{F}}(\R^n\times\R^m)$ and BMO space ${\rm BMO}_{\mathcal{F}}(\R^n\times\R^m)$  were introduced by Han, Lu and Sawyer \cite{HL,HLS} 
through their creation of a flag type Littlewood--Paley theory. More recently, Han, Lee, and the second and fifth 
authors \cite{HLLW} established a full characterization of $H^1_{\mathcal{F}}(\R^n\times\R^m)$ via appropriate flag type non-tangential, radial maximal functions,
Littlewood--Paley theory via Poisson integrals, the flag type Riesz transforms, as well as flag atomic decompositions.

{{In the multi-parameter setting, the dilation structure $\delta \cdot x=(\delta
_{1}x_{1},\ldots,\delta _{n}x_{n})$, for $\delta :=(\delta _{1},\ldots,\delta
_{n})\in (0,\infty)^{n}$, determines a geometry that is reflected by axes-parallel rectangles of arbitrary side-lengths. Indeed, the strong maximal function is defined as the supremum of averages over such rectangles, and the Chang--Fefferman product BMO space can also be characterized using such rectangles. When it comes to the flag setting, the lack of an explicit dilation structure makes its geometry much more obscure. However, from the study of properties of the flag singular integrals, such as the flag Riesz transforms that will be introduced below, one realizes that the flag geometry can be reflected by axes-parallel rectangles with certain restriction on the side-lengths. For example, the flag rectangles in $\mathbb{R}^n\times \mathbb{R}^m$ are the ones of the form $R=I\times J\subset \mathbb{R}^n\times \mathbb{R}^m$ with $\ell(I)\leq \ell(J)$. Compared to the multi-parameter setting, the restriction $\ell(I)\leq \ell(J)$ gives rise to new difficulties. For instance, a very useful trick in the study of problems in the multi-parameter setting is to take a sequence of rectangles $\{I\times J_i\}$ and let $J_i$ shrink to a point $y_0$ as $i\to \infty$. This can usually effectively reduce the problem to one-parameter. However, in the flag setting, such operation is not allowed any more. Other intrinsic difficulties of the flag setting can be better described from the analytic perspective, which will be discussed below.}}


A commutator of a classical Cald\'eron--Zygmund singular integral with a ${\rm BMO}$ function is a bounded operator on $L^p$ with norm equivalent to the ${\rm BMO}$ norm of the symbol (\cite{crw}).
Modern methods of proving the upper bound of these commutators in the multi-parameter product setting rely upon the existence of a wavelet basis for $L^2(\R^n)$, such as the Meyer wavelets or Haar wavelets, see for example \cite{LPPW,DO}.  It turns out that the behavior of the commutator is straightforward to analyze in terms of the wavelet basis.  One method of proof shows that the commutator can be written as a linear combination of paraproducts and simple wavelet analogs of the Calder\'on--Zygmund operator in question.  The other approach uses the wavelet basis to dominate the commutator by a composition of sparse operators.  In the flag setting, we lack a suitable wavelet basis and this approach is not available. Essentially, the wavelet basis  requires the construction of a suitable multi-resolution analysis, which we do not have in this flag setting. Hence, instead of the wavelet basis, we resort to using a method based on heat semi-groups and flag type Littlewood--Paley theory, exploiting the connection between the Reisz transforms and the Laplacian.

We now recall the flag Riesz transforms as studied in \cite{HLLW}. We use $R^{(1)}_j$ to denote the $j$-th Riesz transform on $\R^{n+m}$, $j=1,2,\ldots,n+m$, and we use $R^{(2)}_k$ to denote the $k$-th Riesz transform on $\R^m$, $k=1,2,\ldots,m$.  Namely, we have that for $g^{(1)}\in L^2(\R^{n+m})$,
$$
R^{(1)}_j g^{(1)}(x)={\rm p.v.}\ c_{n+m}\int_{\mathbb{R}^{n+m}} \frac{x_j-y_j}{\left\vert x-y \right\vert^{n+m+1}} g^{(1)}(y)dy,\quad x\in\mathbb R^{n+m}; $$
\text{ and } for $g^{(2)}\in L^2(\R^{m})$,
$$R^{(2)}_k g^{(2)}(z)={\rm p.v.}\ c_m\int_{\mathbb{R}^m} \frac{w_j-z_j}{\left\vert w-z \right\vert^{m+1}} g^{(2)}(w)dw, \quad z\in\mathbb R^{m}.
$$
For $f\in L^2(\mathbb{R}^{n+m})$, we set 
\begin{align}\label{flag Riesz}
R_{j,k}(f)=R^{(1)}_j\ast R^{(2)}_k \ast_2 f,
\end{align}
that is, $R_{j,k}$ is the composition of $R^{(1)}_j$ and $R^{(2)}_k.$ Note that the flag structure appears in $R_{j,k}.$

Given two functions $b, f\in L^2(\mathbb{R}^{n+m})$, 
we first recall the usual definition of commutator
\begin{align}
[b,R_j^{(1)}](f)(x_1,x_2) := b(x_1,x_2) R_j^{(1)} \ast f (x_1,x_2) - R_j^{(1)} \ast (bf)(x_1,x_2).
\end{align}
The commutator can also act only on the second variable:
\begin{align}
[b,R_k^{(2)}]_{2}(f)(x_1,x_2):= b(x_1,x_2) R_k^{(2)} \ast_2 f (x_1,x_2) - R_k^{(2)} \ast_2 (bf)(x_1,x_2).
\end{align}
Iterated commutators arise in the study of commutators of multi-parameter singular integral operators which are 
tensor products. In the flag setting, our iterated commutator takes the following form:
\begin{defn}\label{def-of-flag-commutator}
Given two functions $b, f\in L^2(\mathbb{R}^{n+m})$, the iterated commutator in the flag setting of $\R^n\times\R^m$ is
\begin{align*}
[[b,R_j^{(1)}],R^{(2)}_k]_2(f) &:= b(x_1,x_2) R_j^{(1)} \ast R^{(2)}_k\ast_2 f(x_1,x_2) - R_j^{(1)} \ast (b\cdot R^{(2)}_k\ast_2 f) (x_1,x_2)\\
&\quad - R^{(2)}_k \ast_2 \big( b\cdot R_j^{(1)} \ast f\big) (x_1,x_2) + R^{(2)}_k\ast_2 R_j^{(1)}\ast (b\cdot f) (x_1,x_2).
\end{align*}
\end{defn}
We point out that another possible definition via $[[b,R^{(2)}_k]_2,R_j^{(1)}](f)$ turns out to be equivalent; see Proposition \ref{prop equi com} in Section 2.

We also introduce the {\it big} commutator in the flag setting as 
follows.
\begin{defn}\label{def-of-big-commutator}
Given two functions $b, f\in L^2(\mathbb{R}^{n+m})$, the big commutator in the flag setting of $\R^n\times\R^m$ is
\begin{align}
[b,R_{j,k}](f)(x):= b(x) R_{j,k}(f)(x) - R_{j,k}(bf)(x).
\end{align}
\end{defn}

The main results, below, of this paper relate iterated and big commutator bounds to flag ${\rm BMO}$ spaces. As the definition of the space ${\rm BMO}_{\mathcal{F}}(\R^n\times \R^m)$ is very technical, we refer the reader to Section \ref{s2}, Definition \ref{def flag BMO} for details.
\begin{thm}\label{thm iterated com}
Suppose $b\in {\rm BMO}_{\mathcal{F}}(\R^n\times\R^m)$ and $1<p<\infty$. Then for every $j=1,\ldots,n+m$, $k=1,\ldots,m$, $f\in L^p(\R^{n+m})$,
\begin{align}
\|[[b,R_j^{(1)}],R^{(2)}_k]_2(f)\|_{L^p(\R^{n+m})}\ls \|b\|_{{\rm BMO}_{\mathcal{F}}(\mathbb{R}^{n}\times \mathbb{R}^{m})}\|f\|_{L^p(\R^{n+m})}.
\end{align}
\end{thm}
Lacking methods related to analyticity (\cite{FS} for the Hilbert transform) or wavelets (\cite{LPPW,LPPW2}, \cite{DO}), we instead obtain this upper bound using the duality argument and the tools of semigroups, harmonic function extensions and techniques from  multi-parameter analysis.

Next, we introduce the little flag BMO space. The flag structure has a geometry which is reflected by the axes-parallel rectangles $R=I\times J\subset\mathbb{R}^{n+m}$ satisfying $\ell(I)\leq \ell(J)$, the collection of which is referred to as \emph{flag rectangles}, denoted by $\mathcal{R}_\mathcal{F}$. One can then define the little flag BMO space and the flag type Muckenhoupt   weights $A_{{\mathcal F},p}$  with respect to $\mathcal{R}_\mathcal{F}$. 
\begin{defn}\label{def-of-hardy-by-han}
A locally integrable function $b$ is in \emph{little flag BMO space}, denoted by $\text{bmo}_{\mathcal{F}}(\R^{n}\times\R^{m})$, if
\begin{equation}\label{bmodef}
\|b\|_{\text{bmo}_{\mathcal{F}}(\R^{n}\times\R^{m})}:=\sup_{R\in\mathcal{R}_{\mathcal{F}}}{1\over |R|}\int_R|b(x,y)-\langle b\rangle_R|\,dxdy<\infty,
\end{equation}
where $\langle b\rangle_R = {1\over |R|}\int_R b(x_1,x_2)\,dx_1dx_2$.
\end{defn}

\begin{thm}\label{thm big com}
Suppose $T_\mathcal{F}$ is a flag singular integral operator on $\mathbb{R}^n\times\mathbb{R}^m$, $b\in {\rm bmo}_{\mathcal{F}}(\R^n\times\R^m)$ and $1<p<\infty$. Then for 
$f\in L^p(\R^{n+m})$,
\begin{align}
\|[b,T_\mathcal{F}](f)\|_{L^p(\R^{n+m})}\ls \|b\|_{ {\rm bmo}_{\mathcal{F}}(\R^n\times\R^m)  } \|f\|_{L^p(\R^{n+m})}.
\end{align}
\end{thm}
In the above, the flag singular integral $T_\mathcal{F}$ can be taken as the Riesz transform $R_{j,k}$. The class of flag singular integral operators $T_\mathcal{F}$ naturally generalize the Riesz transforms $R_{j,k}$ and are assumed to be associated to kernels having a standard flag structure. We refer the reader to Definition \ref{FlagSIO} in Section 4 for its precise definition. To obtain this upper bound, we study the little flag BMO space ${\rm bmo}_{\mathcal{F}}(\R^n\times\R^m)$
and find the connection with the John--Nirenberg BMO space on $\R^{n+m}$ and on $\R^m$. We also establish the bridge between functions in ${\rm bmo}_{\mathcal{F}}(\R^n\times\R^m)$ and
weights in $A_{{\mathcal F},p}$. These structures lead to the upper bound for $[b,R_{j,k}](f)$.

As application, the commutator estimates obtained above imply certain versions of div-curl lemmas, which seem to be first of their kind in the flag setting. Roughly speaking, a div-curl lemma says that if vector fields $E$ and $B$ initially in $L^2$ have some cancellation (e.g. divergence or curl zero) then one can expect their dot product $E\cdot B$ to belong to a better space of functions instead of just $L^1$ (as provided for by Cauchy-Schwarz). The cancellation conditions allow one to deduce some type of cancellation, e.g. $\int E\cdot B =0$, suggesting that the function should belong to a suitable Hardy space since it is integrable and has mean zero. The algebraic structure of $E\cdot B$ coupled with the duality between Hardy spaces and BMO spaces then points to the use of the commutator theorem to arrive at the membership of $E\cdot B$ in the Hardy space; different commutator results suggest different div-curl lemmas that can be explored. In the classical one-parameter setting, the div-curl lemma says that given two vector fields, one with divergence zero and the other with curl zero, their dot product belongs to a Hardy space \cite{CLMS}. Later on, Lacey, Petermichl, and the fourth and the fifth authors proved multiple versions of div-curl lemmas in the multi-parameter setting \cite{LPPW3}, which are expected since the multi-parameter setting offers several different interpretations of the Hardy and BMO spaces. Thus, it is natural that our Theorems \ref{thm iterated com} and \ref{thm big com} lead to two versions of flag type div-curl lemmas. 

First, consider vector fields on $\R^{n}\times\R^m$ that take values in $\mathcal M_{{n+m},m}$ and are associated with the flag structure (see Section 5 for the precise definitions and details). We establish the div-curl lemma in the flag setting with respect to the flag Hardy space below, which is a consequence of Theorem \ref{thm iterated com}. 
\begin{thm}\label{thm divcurl1}
Let $1<p,q<\infty$ with ${1\over p}+{1\over q}=1$. Suppose that 
$E,B$ are vector fields on $\R^{n}\times\R^m$ taking the values in $\mathcal M_{{n+m},m}$, associated  with the flag structure. Moreover, suppose
$E=E^{(1)}\ast_2 E^{(2)}\in L^p_{\mathcal{F}}(\R^{n}\times\R^m; \mathcal M_{{n+m},m})$ and 
$B=B^{(1)}\ast_2 B^{(2)}\in L^q_{\mathcal{F}}(\R^{n}\times\R^m; \mathcal M_{{n+m},m})$ satisfy that
$$ \operatorname{div}_{(x,y)} E^{(1)}_{j}(x,y)=0\quad{\rm and}\quad \operatorname{curl}_{(x,y)} B^{(1)}_{j}(x,y)=0,\quad   \forall k $$
and
$$ \operatorname{div}_{y} E^{(2)}_{k}(x,y)=0\quad{\rm and}\quad \operatorname{curl}_{y} B^{(2)}_{k}(x,y)=0  ,\quad \forall x\in\R^n,\ \forall j.$$
Then $E\cdot B$ belongs to the flag Hardy space $H^1_{\mathcal{F}}(\R^n\times\R^m)$ with
\begin{align}
\|E\cdot B\|_{H^1_{\mathcal{F}}(\R^n\times\R^m)}\ls \|E\|_{L^p(\R^{n}\times\R^m; \mathcal M_{{n+m},m})}
\|B\|_{L^q(\R^{n}\times\R^m; \mathcal M_{{n+m},m})}.
\end{align}
\end{thm}
We also prove another version of the div-curl lemma in the flag setting, which is with respect to the Hardy spaces on $\mathbb R^{n+m}$ and on $\mathbb R^m$, respectively. This
version relies on the intermediate result in the proof of Theorem \ref{thm big com}, namely, the structure of the flag little bmo space. 
\begin{thm}\label{thm divcurl2}
Let $1<p,q<\infty$ with ${1\over p}+{1\over q}=1$. Suppose that 
$E,B$ are vector fields on $\R^{n}\times\R^m$ taking the values in $\R^{n+m}$. Moreover, suppose
$E\in L^p(\R^{n}\times\R^m; \R^{n+m})$ and 
$B\in L^q(\R^{n}\times\R^m; \R^{n+m})$ satisfy that
$$ \operatorname{div}_{(x,y)} E(x,y)=0\quad{\rm and}\quad \operatorname{curl}_{(x,y)} B(x,y)=0 $$
and
$$ \operatorname{div}_{y} E(x,y)=0\quad{\rm and}\quad \operatorname{curl}_{y} B(x,y)=0  ,\quad \forall x\in\R^n.$$
Then we have
\begin{align}
\|E\cdot B\|_{H^1(\R^{n+m})}\ls \|E\|_{L^p(\R^{n}\times\R^m; \R^{n+m})}
\|B\|_{L^q(\R^{n}\times\R^m; \R^{n+m})}.
\end{align}
and
\begin{align}
\int_{\R^m}\|E(\cdot,y)\cdot_2 B(\cdot,y)\|_{H^1(\R^{m})}\,dy \ls \|E\|_{L^p(\R^{n}\times\R^m; \R^{n+m})}
\|B\|_{L^q(\R^{n}\times\R^m; \R^{n+m})},
\end{align}
where $$ E(x,y)\cdot_2 B(x,y) := \sum_{k=1}^m E_{n+k}(x,y) B_{k}(x,y). $$
\end{thm}

It is known that the div-curl lemma in the classical setting has many applications in PDE and compensated compactness \cite{CLMS}. Similarly, we expect that the flag type div-curl lemmas described above would have interesting implications in these directions as well. For instance, following the ideas in \cite{CLMS}, one can study weak convergence problems in the flag Hardy space. And it would be interesting to know whether one can use the flag type regularity (implied by our div-curl lemmas) of certain nonlinear quantities to obtain improved regularity results for certain nonlinear PDE.

This paper is organised as follows. 
In Section 2 we provide necessary preliminaries with respect to the flag structures. 
In Section 3 we study the flag iterated commutators as in Definition \ref{def-of-flag-commutator}
and prove Theorem \ref{thm iterated com}. In Section 4 we give a complete treatment of the 
flag little bmo spaces and flag type Muckenhoupt $A_p$ weights, toward the proof of Theorem \ref{thm big com}.
In the last section, we apply the boundedness of flag commutators
from Theorems \ref{thm iterated com} and \ref{thm big com} to establish the flag 
div-curl results, Theorems  \ref{thm divcurl1} and \ref{thm divcurl2}.

\section{Preliminaries in the flag setting}\label{s2}
\setcounter{equation}{0}

Recall the classical Poisson kernel on $\mathbb R^n$:
$$ P(x) := {c_n\over (1+|x|^2)^{n+1\over2}}. $$
And we define
$$ P_t(x) := {1\over t^n} P({x\over t}). $$
For $f\in L^1(\mathbb R^n)$, let $F(x,t) := P_t*f(x)$.
Then we have the following standard pointwise estimates for the Poisson integral, see in particular Stein (\cite{s93}).
\begin{prop}Suppose $f\in L^1(\mathbb R^n)$. Then
\begin{align}
\sup_{(x,t)\in\mathbb R^{n+1}_+} t^{n+k} |\nabla^k_{x,t} F(x,t)| \leq C \|f\|_{L^1(\mathbb R^n)}.
\end{align}
\end{prop}

We now recall the flag Poisson kernel given by
\begin{eqnarray*}
 P(x,y)=P^{(1)}\ast_{\mathbb R^m}
P^{(2)}(x,y)=\int_{\mathbb{R}^m}P^{(1)}(x,y-z)P^{(2)}(z)dz
\end{eqnarray*}
where
\begin{eqnarray*}
P^{(1)}(x,y)=\frac{\displaystyle c_{n+m}} {\displaystyle
(1+|x|^2+|y|^2)^{(n+m+1)/2} }\ \ {\rm and}\ \
P^{(2)}(z)=\frac{\displaystyle c_{m}} {\displaystyle
(1+|z|^2)^{(m+1)/2} }\ \ \ \
\end{eqnarray*}
are the classical Poisson kernels on $\mathbb{R}^{n+m}$ and $\mathbb{R}^m$,
respectively. Then we have
$$ P_{t_1,t_2}(x,y)=P^{(1)}_{t_1}\ast_{\mathbb R^m}
P^{(2)}_{t_2}(x,y).$$

We define the Lusin area function with respect to
$u = P_{t_1,t_2}\ast f$ as follows.
\begin{defn}\label{def-of-S-function-Su}
For $f\in L^1(\mathbb{R}^n\times\mathbb{R}^m)$ and $u(x_1,x_2,t_1,t_2)=P_{t_1,t_2}\ast f(x_1,x_2),$
$S_{\mathcal{F}}(u)$, the Lusin area integral of $u(x_1,x_2,t_1,t_2)$ is
defined by
\begin{eqnarray}\label{S-function-Su}
&&S_{\mathcal{F}}(u)(x_1,x_2)\\
&&=\bigg\{\int_{\mathbb{R}_{+}^{n+1}}\int_{\mathbb{R}_{+}^{m+1}}\chi_{t,s}(x_1-w_1,x_2-w_2)
|t_1\nabla^{(1)}t_2\nabla^{(2)}u(w_1,w_2,t_1,t_2)|^2 {dw_1dt_1 \over t_1^{n+m+1}}{dw_2dt_2
\over t_2^{m+1}} \bigg\}^{{1\over 2}},\nonumber
\end{eqnarray}
where 
$\nabla^{(1)}=\big(\partial_{t_1},\partial_{w_{1,1}}\cdots\partial_{w_{1,n}},\partial_{w_{2,1}}\cdots\partial_{w_{2,m}}\big)$ is the standard gradient on 
$\mathbb R^{n+m+1}$,
and
$\nabla^{(2)}=\big(\partial_{t_2},\partial_{w_{2,1}}\cdots\partial_{w_{2,m}}\big)$ is the standard gradient on $\mathbb R^{m+1}$, and
 \begin{align}\label{chi function}
 \chi_{t_1,t_2}(x_1,x_2):=\chi_{t_1}^{(1)}\ast_{\mathbb R^m} \chi_{t_2}^{(2)}(x_1,x_2),
 \end{align}
$\chi_{t_1}^{(1)}(x_1,x_2):={t_1}^{-(n+m)}\chi^{(1)}({x_1\over t_1},{x_2\over t_1})$,
$\chi_{t_2}^{(2)}(z):={t_2}^{-m}\chi^{(2)}({z\over t_2})$, $\chi^{(1)}(x,y)$
and $\chi^{(2)}(z)$ are the indicator function of the unit balls of
$\mathbb{R}^{n+m}$ and $\mathbb{R}^m$, respectively.
\end{defn}

\begin{defn}\label{def flag Hardy}
The flag Hardy space $H_{\mathcal{F}}^1(\mathbb{R}^n\times\mathbb{R}^m)$
is defined to be the collection of  $f\in
L^1(\mathbb{R}^n\times\mathbb{R}^m) $ such that $ S_{\mathcal{F}}(u)\in
L^1(\mathbb{R}^n\times\mathbb{R}^m) $. The norm of $H_{\mathcal{F}}^1(\mathbb{R}^n\times\mathbb{R}^m)$  is
defined by
\begin{eqnarray}\label{Hp norm}
\|f\|_{H_{\mathcal{F}}^1(\mathbb{R}^n\times\mathbb{R}^m)}=\|S_{\mathcal{F}}(u)\|_{L^1(\mathbb{R}^n\times\mathbb{R}^m)}.
\end{eqnarray}
\end{defn}

We now recall the definition of the flag BMO space.
\begin{defn}\label{def flag BMO}
The flag BMO space ${\rm BMO}_{\mathcal{F}}(\mathbb{R}^n\times\mathbb{R}^m)$
is defined to be the collection of  $b\in
L_{loc}^1(\mathbb{R}^n\times\mathbb{R}^m) $ such that 
\begin{align}\label{BMO norm}
\|b\|_{{\rm BMO}_{\mathcal{F}}(\mathbb{R}^n\times\mathbb{R}^m)}:=\sup_{\Omega}\bigg( {1\over |\Omega|} 
\int_{T(\Omega)} |t_1\nabla^{(1)}t_2\nabla^{(2)}u(w_1,w_2,t_1,t_2)|^2{dw_1dt_1 dw_2dt_2\over t_1t_2}\bigg)^{1\over2}<\infty,
\end{align}
where the supremum is taken over all open sets in $\mathbb{R}^n\times\mathbb{R}^m$  with finite measures, and
$T(\Omega)=\cup_{R\subset \Omega} T( R)$
with the rectangle $R=I\times J$, $\ell(I)\leq\ell(I)$ and $T(R)= I\times ({\ell(I)\over2},\ell(I)]\times J\times ({\ell(J)\over2},\ell(J)]$. 
\end{defn}

\begin{prop}\label{prop equi com}
Given two functions $b, f\in L^2(\mathbb{R}^{n+m})$, we have
\begin{align}\label{equi com}
[[b,R_j^{(1)}],R^{(2)}_k]_2(f)= [[b,R^{(2)}_k]_2,R_j^{(1)}](f).
\end{align}
\end{prop}
\begin{proof}
By definition, we see that
\begin{align*}
[[b,R_j^{(1)}],R^{(2)}_k]_2(f)(x_1,x_2) &= [b,R_j^{(1)}] R^{(2)}_k\ast_2 f(x_1,x_2)  - R^{(2)}_k\ast_2 ( [b,R_j^{(1)}](f) )(x_1,x_2) \\
&= b(x_1,x_2) R_j^{(1)} \ast R^{(2)}_k\ast_2 f(x_1,x_2) - R_j^{(1)} \ast (b\cdot R^{(2)}_k\ast_2 f) (x_1,x_2)\\
&\quad - R^{(2)}_k \ast_2 \big( b\cdot R_j^{(1)} \ast f - R_j^{(1)}\ast (b\cdot f) \big)(x_1,x_2)\\
&= b(x_1,x_2) R_j^{(1)} \ast R^{(2)}_k\ast_2 f(x_1,x_2) - R_j^{(1)} \ast (b\cdot R^{(2)}_k\ast_2 f) (x_1,x_2)\\
&\quad - R^{(2)}_k \ast_2 \big( b\cdot R_j^{(1)} \ast f\big) (x_1,x_2) + R^{(2)}_k\ast_2 R_j^{(1)}\ast (b\cdot f) (x_1,x_2).
\end{align*}
And we also have
\begin{align*}
 [[b,R^{(2)}_k]_2,R_j^{(1)}](f)(x_1,x_2) &= [b,R^{(2)}_k]_2 R_j^{(1)}\ast f(x_1,x_2)  - R_j^{(1)}\ast ( [b,R^{(2)}_k]_2(f) )(x_1,x_2) \\
&= b(x_1,x_2) R^{(2)}_k\ast_2 R_j \ast  f(x_1,x_2) -R^{(2)}_k\ast_2 (b\cdot  R_j ^{(1)}\ast f)(x_1,x_2)\\
&\quad -  R_j^{(1)} \ast \big( b\cdot R^{(2)}_k \ast_2 f - R_k^{(2)} \ast_2 (b\cdot f) \big)(x_1,x_2)\\
&=  b(x_1,x_2) R^{(2)}_k\ast_2 R_j^{(1)} \ast  f(x_1,x_2) -R^{(2)}_k\ast_2 (b\cdot  R_j^{(1)} \ast f)(x_1,x_2)\\
&\quad - R_j^{(1)}\ast  \big( b\cdot R^{(2)}_k \ast_2 f\big) (x_1,x_2) +   R_j^{(1)}\ast R^{(2)}_k\ast_2 (b\cdot f) (x_1,x_2).
\end{align*}
It is direct to see that, by changing of variables, 
\begin{align*}
  R^{(2)}_k\ast_2 R_j^{(1)} \ast  f(x_1,x_2)&= \int  R^{(2)}_k(x_2-z) R_j^{(1)} (x_1-y_1,z-y_2)   f(y_1,y_2)\ dzdy_1dy_2\\
  &= \int  R^{(2)}_k (\tilde z- y_2) R_j^{(1)} (x_1-y_1,x_2-\tilde z)   f(y_1,y_2)\ d\tilde zdy_1dy_2\\
&= \int  R_j^{(1)} (x_1-y_1,x_2-\tilde z)  R^{(2)}_k (\tilde z- y_2)  f(y_1,y_2)\ d\tilde zdy_1dy_2\\
&=  R_j^{(1)} \ast R^{(2)}_k\ast_2 f(x_1,x_2),
\end{align*}
which implies that \eqref{equi com} holds.
\end{proof}

\section{Upper bound of the iterated commutator $[[b,R^{(1)}_i], R^{(2)}_j]_2$}\label{s3}

In this section, we prove Theorem \ref{thm iterated com}, i.e., the upper bound of the iterated commutator $[[b,R^{(1)}_i], R^{(2)}_j]_2$.
As we pointed out earlier, in the flag setting, there is lack of a suitable wavelet basis or Haar basis and hence the approaches in \cite{LPPW,DO} are not available. We establish a fundamental duality argument (Lemma \ref{lemma Tent}) with respect to general flag type area integrals and flag Carleson measures, and then apply the technique of harmonic expansion to obtain the full versions of flag type Carleson measure inequalities (Proposition \ref{prop BMO estimate}), which plays the role of ``paraproducts''. Then, by considering the bilinear form associated with the  iterated commutator $[[b,R^{(1)}_i], R^{(2)}_j]_2$ and by integration by part, we can decompose the bilinear form into a summation of different versions of ``paraproducts''. Then the upper bound of the iterated commutator $[[b,R^{(1)}_i], R^{(2)}_j]_2$
follows from applying Proposition \ref{prop BMO estimate} to each ``paraproducts''.

\subsection{Extension via flag Poisson operator}

For any $f\in L^1(\mathbb{R}^n\times\mathbb{R}^m)$, we
define the flag Poisson integral of $f$ by
\begin{align}\label{poisson integral}
F(x_1,x_2,t_1,t_2):=P_{t_1,t_2}\ast f(x_1,y_2),
\end{align}
where 
\begin{align}\label{flag poisson}
P_{t_1,t_2}(x_1,x_2)=P^{(1)}_{t_1}\ast_{\mathbb R^m}P^{(2)}_{t_2}(x_1,x_2).
\end{align}

Since $P(x_1,x_2)\in L^1(\mathbb{R}^n\times\mathbb{R}^m)$, it easy to see
that $F(x_1,x_2,t_1,t_2)$ is well-defined. Moreover, for any
fixed $t_1$ and $t_2$, $P_{t_1,t_2}\ast f(x_1,x_2)$ is a bounded $C^{\infty}$
function and the
function $F(x_1,x_2,t_1,t_2)$ is harmonic
in $(x_1,x_2,t_1)$ and $(x_2,t_2)$, respectively.  $F(x_1,x_2,t_1,t_2)$ is the flag harmonic extension of $f$ to $\mathbb R^{n+1}_+\times \mathbb R^{m+1}_+$. More precisely,
\begin{align}\label{extension}
\left\{                 
\begin{aligned}
& \Delta_{\mathbb R^{n+m+1} } F(x_1,x_2,t_1,t_2) = (\partial_{t_1}^2 +\Delta_{x_1,x_2})  F(x_1,x_2,t_1,t_2)  =0 \quad {\rm in}\ \mathbb R^{n+m+1}_+; \\
& \Delta_{\mathbb R^{m+1} } F(x_1,x_2,t_1,t_2) = (\partial_{t_2}^2 +\Delta_{x_2})  F(x_1,x_2,t_1,t_2)  =0 \quad {\rm in}\ \mathbb R^{m+1}_+;\\
& \lim_{t_1\to0} \partial_{t_1} F(x_1,x_2,t_1,t_2) = -(\Delta_{x_1,x_2})^{1\over2} P^{(2)}\ast_{\mathbb R^m}f(x_1,x_2) \quad {\rm on}\ \mathbb R^{n+m};\\
& \lim_{t_2\to0} \partial_{t_2} F(x_1,x_2,t_1,t_2) = -(\Delta_{x_2})^{1\over2} P^{(1)}\ast f(x_1,x_2) \quad {\rm on}\ \mathbb R^{n+m};\\
& \lim_{t_1\to0}  F(x_1,x_2,t_1,t_2) = P^{(2)}\ast_{\mathbb R^m}f(x_1,x_2) \quad {\rm on}\ \mathbb R^{n+m};\\
& \lim_{t_2\to0}  F(x_1,x_2,t_1,t_2) = P^{(1)}\ast f(x_1,x_2) \quad {\rm on}\ \mathbb R^{n+m};\\
& \lim_{t_1\to0,\ t_2\to0}  F(x_1,x_2,t_1,t_2) = f(x_1,x_2) \quad {\rm on}\ \mathbb R^{n+m};\\
& \lim_{|(x_1,x_2,t_1)|\to\infty}  F(x_1,x_2,t_1,t_2) = 0;\\
& \lim_{|(x_2,t_2)|\to\infty}  F(x_1,x_2,t_1,t_2) = 0.
\end{aligned}
\right.
\end{align}

We then have the following lemma providing a connection between the boundary values $f$ and the flag harmonic extension $F$.  This follows from the decay of the flag harmonic extensions of $f$ and repeated applications of integration by parts in the variables $t_1$ and $t_2$.
\begin{lem}
For $f\in L^1(\mathbb{R}^n\times\mathbb{R}^m)$, let $F$ be the same as in \eqref{poisson integral}. Then we have
\begin{align}\label{extension integration}
\int_{\mathbb R^n\times\mathbb R^m} f(x_1,x_2)dx_1dx_2 = \int_{\mathbb R^{n+1}_+\times\mathbb R^{m+1}_+} t_1 \partial_{t_1}^2t_2 \partial_{t_2}^2 F(x_1,x_2,t_1,t_2)dx_1dx_2dt_1dt_2. 
\end{align}
\end{lem}
\begin{proof}
We start from the right-hand side of \eqref{extension integration}. We write
\begin{align*}
& \int_{\mathbb R^{n+1}_+\times\mathbb R^{m+1}_+} t_1 \partial_{t_1}^2t_2 \partial_{t_2}^2 F(x_1,x_2,t_1,t_2)dx_1dx_2dt_1dt_2 \\
&=    \int_{\mathbb R^{m+1}_+}   t_2 \partial_{t_2}^2P^{(2)}_{t_2}  \ast_{\mathbb R^m}  \bigg(\int_{\mathbb R^{n+1}_+}    t_1 \partial_{t_1}^2 P^{(1)}_{t_1} * f(x_1,x_2) dx_1dt_1  \bigg) dx_2dt_2\\
&=    \int_{\mathbb R^{m}}    \bigg(\int_{\mathbb R^{n+1}_+}    t_1 \partial_{t_1}^2 P^{(1)}_{t_1} * f(x_1,x_2) dx_1dt_1  \bigg) dx_2,
\end{align*}
where the last equality follows from  decay of the flag harmonic extensions of $f$ and using the integration by part in the variables $t_2$.
To continue, we write the right-hand side of the last equality above as 
\begin{align*}
&     \int_{\mathbb R^{n+m+1}_+}    t_1 \partial_{t_1}^2 P^{(1)}_{t_1} * f(x_1,x_2) dx_1 dx_2dt_1
= \int_{\mathbb R^{n+m}}     f(x_1,x_2) dx_1 dx_2,
\end{align*}
which yields \eqref{extension integration}. Again, the last equality follows from  decay of the flag harmonic extensions of $f$ and using the integration by part in the variables $t_1$.
\end{proof}

\subsection{Flag area functions and estimates}

We also have a more general version of the area function.
\begin{defn}\label{def-of-S-function-SG}
For a function $G(x_1,x_2,t_1,t_2)$ defined on  $\mathbb{R}_{+}^{n+1}\times \mathbb{R}_{+}^{m+1}$, the general flag type Lusin area integral of $G$ is 
defined by
\begin{eqnarray}\label{S-function-SG}
&&S_{\mathcal{F},L}(G)(x_1,x_2)\\
&&:=\bigg\{\int_{\mathbb{R}_{+}^{n+1}}\int_{\mathbb{R}_{+}^{m+1}}\chi_{t,s}(x_1-w_1,x_2-w_2)
|G(w_1,w_2,t_1,t_2)|^2 {dw_1dt_1 \over t_1^{n+m+1}}{dw_2dt_2
\over t_2^{m+1}} \bigg\}^{{1\over 2}}.\nonumber
\end{eqnarray}
\end{defn}

\begin{lem}\label{lemma Tent}
Suppose $F(x_1,x_2,t_1,t_2)$ and $G(x_1,x_2,t_1,t_2)$ are defined on  $\mathbb{R}_{+}^{n+1}\times \mathbb{R}_{+}^{m+1}$.  Then the following estimate holds:
\begin{align}\label{Tent}
&\int_{\mathbb{R}_{+}^{n+1}}\int_{\mathbb{R}_{+}^{m+1}} F(x_1,x_2,t_1,t_2) G(x_1,x_2,t_1,t_2)\, dx_1dx_2dt_1dt_2\\
&\leq C \sup_{\Omega\subset\mathbb R^n\times\mathbb R^m} \bigg( {1\over |\Omega|} \int_{T(\Omega)} t_1\,t_2\, |F(y_1,y_2,t_1,t_2)|^2\, dy_1dy_2dt_1dt_2 \bigg)^{1/2} \nonumber\\
&\quad\times\! \int_{\mathbb{R}^{n}}\int_{\mathbb{R}^{m}} \bigg( \int_{\mathbb{R}_{+}^{n+1}}\int_{\mathbb{R}_{+}^{m+1}} \chi_{t_1,t_2}(x_1-y_1,x_2-y_2)  |G(y_1,y_2,t_1,t_2)|^2\, {dy_1dy_2dt_1dt_2\over t_1^{n+m+1}t_2^{m+1}} \bigg)^{1/2} \, dx_1dx_2.\nonumber
\end{align}
\end{lem}
\begin{proof}
Suppose both factors on the right-hand side above are finite, since otherwise there is nothing to prove.
We also note that the second factor is actually $\|S_{\mathcal F}(G)\|_{L^1(\mathbb R^n\times\mathbb R^m)}$.

We now let
$$ \Omega_k :=\{ (x_1,x_2)\in \mathbb{R}^n\times\mathbb{R}^m:\  S_{\mathcal{F},L}(G)(x_1,x_2) > 2^k\} $$
and define
$$ B_k:= \{ R=I_1\times I_2:\ |(I_1\times I_2)\cap \Omega_k|> {1\over 2}|I_1\times I_2|,\ |(I_1\times I_2)\cap \Omega_{k+1}|\leq {1\over 2}|I_1\times I_2| \},$$
where $I_1$ and $I_2$ are dyadic cubes in $\mathbb R^n$ and $\mathbb R^m$ with side-lengths $\ell(I)$ and $\ell(J)$ satisfying $\ell(I)\leq \ell(J)$. 
Moreover, we define
$$ \Omega_k=\bigcup_{R\in B_k} R $$
and
$$ \widetilde\Omega_k=\Big\{(x_1,x_2)\in \mathbb R^n\times\mathbb R^m:\ M_{flag}(\chi_{\Omega_k})(x_1,x_2)>{1\over 2}\Big\}. $$
Next, we have 
\begin{align*}
&\int_{\mathbb{R}_{+}^{n+1}}\int_{\mathbb{R}_{+}^{m+1}} F(x_1,x_2,t_1,t_2) G(x_1,x_2,t_1,t_2)\, dx_1dx_2dt_1dt_2\\
&=\sum_k \sum_{R\in B_k}\int_{T(R)} \sqrt{t_1t_2}F(x_1,x_2,t_1,t_2) \, {G(x_1,x_2,t_1,t_2)\over \sqrt{t_1t_2}}dx_1dx_2dt_1dt_2\\
&\leq\sum_k \bigg(\sum_{R\in B_k}\int_{T(R)} t_1t_2|F(x_1,x_2,t_1,t_2)|^2 dx_1dx_2dt_1dt_2\bigg)^{1/2}\\ 
&\quad\quad\quad\times \bigg(\sum_{R\in B_k}\int_{T(R)} |G(x_1,x_2,t_1,t_2)|^2\, {dx_1dx_2dt_1dt_2\over t_1t_2}\bigg)^{1/2}\\
&=\sum_k \bigg({1\over |\Omega_k|}\sum_{R\in B_k}\int_{T(R)} t_1t_2|F(x_1,x_2,t_1,t_2)|^2 dx_1dx_2dt_1dt_2\bigg)^{1/2}\\ 
&\quad\quad\quad\times \bigg(|\Omega_k|\sum_{R\in B_k}\int_{T(R)} |G(x_1,x_2,t_1,t_2)|^2\, {dx_1dx_2dt_1dt_2\over t_1t_2}\bigg)^{1/2}\\
&\leq\sum_k \bigg({1\over |\Omega_k|}\int_{T(\Omega_k)} t_1t_2|F(x_1,x_2,t_1,t_2)|^2 dx_1dx_2dt_1dt_2\bigg)^{1/2}\\ 
&\quad\quad\quad\times \bigg(|\widetilde\Omega_k|\sum_{R\in B_k}\int_{T(R)} |G(x_1,x_2,t_1,t_2)|^2\, {dx_1dx_2dt_1dt_2\over t_1t_2}\bigg)^{1/2}\\
&\leq \sup_{\Omega\subset\mathbb R^n\times\mathbb R^m} \bigg({1\over |\Omega|}\int_{T(\Omega)} t_1t_2|F(x_1,x_2,t_1,t_2)|^2 dx_1dx_2dt_1dt_2\bigg)^{1/2}\\ 
&\quad\quad\quad\times \sum_k \bigg(|\widetilde\Omega_k|\sum_{R\in B_k}\int_{T(R)} |G(x_1,x_2,t_1,t_2)|^2\, {dx_1dx_2dt_1dt_2\over t_1t_2}\bigg)^{1/2}.
 \end{align*}

As for the second factor in the last inequality above, note that
\begin{align*}
&  2^{2k} |\widetilde\Omega_k\backslash \Omega_k|\\
&\geq\int_{\widetilde\Omega_k\backslash \Omega_k} S_{\mathcal{F},L}(G)(x_1,x_2)^2  \, dx_1dx_2\\
&=\int_{\widetilde\Omega_k\backslash \Omega_k}  \int_{\mathbb{R}_{+}^{n+1}}\int_{\mathbb{R}_{+}^{m+1}} \chi_{t_1,t_2}(x_1-y_1,x_2-y_2)  |G(y_1,y_2,t_1,t_2)|^2\, {dy_1dy_2dt_1dt_2\over t_1^{n+m+1}t_2^{m+1}}  \, dx_1dx_2\\
&= \int_{\mathbb{R}_{+}^{n+1}}\int_{\mathbb{R}_{+}^{m+1}} \int_{\widetilde\Omega_k\backslash \Omega_k} \chi_{t_1,t_2}(x_1-y_1,x_2-y_2) \, dx_1dx_2\, |G(y_1,y_2,t_1,t_2)|^2\, {dy_1dy_2dt_1dt_2\over t_1^{n+m+1}t_2^{m+1}} \\
&\approx \int_{\mathbb{R}_{+}^{n+1}}\int_{\mathbb{R}_{+}^{m+1}}  |G(y_1,y_2,t_1,t_2)|^2\, {dy_1dy_2dt_1dt_2\over t_1t_2} \\
&\geq \sum_{R\in B_k}\int_{T(R)} |G(x_1,x_2,t_1,t_2)|^2\, {dx_1dx_2dt_1dt_2\over t_1t_2}.
\end{align*}

Thus, we have
\begin{align*}
&\int_{\mathbb{R}_{+}^{n+1}}\int_{\mathbb{R}_{+}^{m+1}} F(x_1,x_2,t_1,t_2) G(x_1,x_2,t_1,t_2)\, dx_1dx_2dt_1dt_2\\
&\leq \sup_{\Omega\subset\mathbb R^n\times\mathbb R^m} \bigg({1\over |\Omega|}\int_{T(\Omega)} t_1t_2|F(x_1,x_2,t_1,t_2)|^2 dx_1dx_2dt_1dt_2\bigg)^{1/2}\\ 
&\quad\quad\quad\times \sum_k \bigg(|\widetilde\Omega_k|2^{2k} |\widetilde\Omega_k\backslash \Omega_k|\bigg)^{1/2}\\
&\leq \sup_{\Omega\subset\mathbb R^n\times\mathbb R^m} \bigg({1\over |\Omega|}\int_{T(\Omega)} |t_1t_2F(x_1,x_2,t_1,t_2)|^2 {dx_1dx_2dt_1dt_2\over t_1t_2}\bigg)^{1/2}\times \sum_k |\Omega_k|2^{k}\\
&\leq \sup_{\Omega\subset\mathbb R^n\times\mathbb R^m} \bigg({1\over |\Omega|}\int_{T(\Omega)} |t_1t_2F(x_1,x_2,t_1,t_2)|^2 {dx_1dx_2dt_1dt_2\over t_1t_2}\bigg)^{1/2}\times \|S_{\mathcal{F},L}(G)\|_{L^1(\mathbb R^n\times\mathbb R^m)},
 \end{align*}
which gives \eqref{Tent}.  This completes the proof of the Lemma \ref{lemma Tent}.
\end{proof}

From Lemma \ref{lemma Tent} above  and the definition of ${\rm BMO}_{\mathcal{F}}(\mathbb{R}^n\times \mathbb{R}^m)$, we can obtain the following Corollary  immediately.
\begin{cor}\label{cor Tent}
Suppose  $G(x_1,x_2,t_1,t_2)$ is defined on  $\mathbb{R}_{+}^{n+1}\times \mathbb{R}_{+}^{m+1}$, and 
$F(x_1,x_2,t_1,t_2):= P_{t_1,t_2}\ast f(x_1,x_2)$, where $f\in {\rm BMO}_{\mathcal{F}}(\mathbb{R}^{n}\times \mathbb{R}^{m})$.  Then we have:
\begin{align}\label{BMO Tent}
&\int_{\mathbb{R}_{+}^{n+1}}\int_{\mathbb{R}_{+}^{m+1}} |\nabla^{(1)}\nabla^{(2)}F(x_1,x_2,t_1,t_2)|\,| G(x_1,x_2,t_1,t_2)|\, dx_1dx_2dt_1dt_2\\
&\leq C \|f\|_{{\rm BMO}_{\mathcal{F}}(\mathbb{R}^{n}\times \mathbb{R}^{m})} \|S_{\mathcal{F},L}(G)\|_{L^1(\mathbb{R}^{n}\times \mathbb{R}^{m})}.\nonumber
\end{align}
\end{cor}
Moreover, based on Lemma \ref{lemma Tent}, we can also establish the following estimates.
\begin{prop}\label{prop BMO estimate}
Suppose $F(x_1,x_2,t_1,t_2)=P_{t_1,t_2}*f(x_1,x_2)$, $G(x_1,x_2,t_1,t_2)=P_{t_1,t_2}*g(x_1,x_2)$, and $B(x_1,x_2,t_1,t_2)=P_{t_1,t_2}*b(x_1,x_2)$. Then we have
\begin{enumerate}
\item \begin{align}\label{BMO 1}
& \int_{\mathbb R^{n+1}_+\times\mathbb R^{m+1}_+} t_1 t_2  | \nabla^{(1)} \nabla^{(2)} B(x_1,x_2,t_1,t_2) |\ |\nabla_{x_1,x_2} \nabla_{x_2}   \nabla^{(1)} \nabla^{(2)} G(x_1,x_2,t_1,t_2) |\\
&\quad\times | \nabla^{(1)} \nabla^{(2)} F(x_1,x_2,t_1,t_2) |  dx_1dx_2dt_1dt_2 \nonumber\\
&\leq C \|b\|_{{\rm BMO}_{\mathcal{F}}(\mathbb{R}^{n}\times \mathbb{R}^{m})}\| (-\Delta_{x_1,x_2})^{1\over2}(-\Delta_{x_2})^{1\over2} g \|_{L^p(\mathbb{R}^{n}\times \mathbb{R}^{m})}\| (-\Delta_{x_1,x_2})^{1\over2}(-\Delta_{x_2})^{1\over2} f \|_{L^{p'}(\mathbb{R}^{n}\times \mathbb{R}^{m})};\nonumber
\end{align}
\item \begin{align}\label{BMO 2}
& \int_{\mathbb R^{n+1}_+\times\mathbb R^{m+1}_+} t_1 t_2  | \nabla^{(1)} \nabla^{(2)} B(x_1,x_2,t_1,t_2) |\ |\nabla_{x_1,x_2} \nabla_{x_2}   \nabla^{(1)} \nabla^{(2)} G(x_1,x_2,t_1,t_2) |\\
&\quad\times | \nabla^{(1)}  F(x_1,x_2,t_1,t_2) |  dx_1dx_2dt_1dt_2 \nonumber\\
&\leq C \|b\|_{{\rm BMO}_{\mathcal{F}}(\mathbb{R}^{n}\times \mathbb{R}^{m})}\| (-\Delta_{x_1,x_2})^{1\over2}(-\Delta_{x_2})^{1\over2} g \|_{L^p(\mathbb{R}^{n}\times \mathbb{R}^{m})}\| (-\Delta_{x_1,x_2})^{1\over2} f \|_{L^{p'}(\mathbb{R}^{n}\times \mathbb{R}^{m})};\nonumber
\end{align}
\item \begin{align}\label{BMO 3}
& \int_{\mathbb R^{n+1}_+\times\mathbb R^{m+1}_+} t_1 t_2  | \nabla^{(1)} \nabla^{(2)} B(x_1,x_2,t_1,t_2) |\ |\nabla_{x_1,x_2} \nabla_{x_2}   \nabla^{(1)} \nabla^{(2)} G(x_1,x_2,t_1,t_2) |\\
&\quad\times |  \nabla^{(2)} F(x_1,x_2,t_1,t_2) |  dx_1dx_2dt_1dt_2 \nonumber\\
&\leq C \|b\|_{{\rm BMO}_{\mathcal{F}}(\mathbb{R}^{n}\times \mathbb{R}^{m})}\| (-\Delta_{x_1,x_2})^{1\over2}(-\Delta_{x_2})^{1\over2} g \|_{L^p(\mathbb{R}^{n}\times \mathbb{R}^{m})}\| (-\Delta_{x_2})^{1\over2} f \|_{L^{p'}(\mathbb{R}^{n}\times \mathbb{R}^{m})};\nonumber
\end{align}
\item 
\begin{align}\label{BMO 4}
& \int_{\mathbb R^{n+1}_+\times\mathbb R^{m+1}_+} t_1 t_2  | \nabla^{(1)} \nabla^{(2)} B(x_1,x_2,t_1,t_2) |\ |\nabla_{x_1,x_2}    \nabla^{(1)} \nabla^{(2)} G(x_1,x_2,t_1,t_2) |\\
&\quad\times |  F(x_1,x_2,t_1,t_2) |  dx_1dx_2dt_1dt_2 \nonumber\\
&\leq C \|b\|_{{\rm BMO}_{\mathcal{F}}(\mathbb{R}^{n}\times \mathbb{R}^{m})}\| (-\Delta_{x_1,x_2})^{1\over2} g \|_{L^p(\mathbb{R}^{n}\times \mathbb{R}^{m})}\|  f \|_{L^{p'}(\mathbb{R}^{n}\times \mathbb{R}^{m})};\nonumber
\end{align}

\item \begin{align}\label{BMO 5}
& \int_{\mathbb R^{n+1}_+\times\mathbb R^{m+1}_+} t_1 t_2  | \nabla^{(1)} \nabla^{(2)} B(x_1,x_2,t_1,t_2) |\ |\nabla_{x_1,x_2}    \nabla^{(1)} \nabla^{(2)} G(x_1,x_2,t_1,t_2) |\\
&\quad\times | \nabla^{(1)} \nabla^{(2)} F(x_1,x_2,t_1,t_2) |  dx_1dx_2dt_1dt_2 \nonumber\\
&\leq C \|b\|_{{\rm BMO}_{\mathcal{F}}(\mathbb{R}^{n}\times \mathbb{R}^{m})}\| (-\Delta_{x_1,x_2})^{1\over2}g \|_{L^p(\mathbb{R}^{n}\times \mathbb{R}^{m})}\| (-\Delta_{x_1,x_2})^{1\over2}(-\Delta_{x_2})^{1\over2} f \|_{L^{p'}(\mathbb{R}^{n}\times \mathbb{R}^{m})};\nonumber
\end{align}
\item \begin{align}\label{BMO 6}
& \int_{\mathbb R^{n+1}_+\times\mathbb R^{m+1}_+} t_1 t_2  | \nabla^{(1)} \nabla^{(2)} B(x_1,x_2,t_1,t_2) |\ |\nabla_{x_1,x_2}    \nabla^{(1)} \nabla^{(2)} G(x_1,x_2,t_1,t_2) |\\
&\quad\times | \nabla^{(1)}  F(x_1,x_2,t_1,t_2) |  dx_1dx_2dt_1dt_2 \nonumber\\
&\leq C \|b\|_{{\rm BMO}_{\mathcal{F}}(\mathbb{R}^{n}\times \mathbb{R}^{m})}\| (-\Delta_{x_1,x_2})^{1\over2} g \|_{L^p(\mathbb{R}^{n}\times \mathbb{R}^{m})}\| (-\Delta_{x_1,x_2})^{1\over2} f \|_{L^{p'}(\mathbb{R}^{n}\times \mathbb{R}^{m})};\nonumber
\end{align}
\item \begin{align}\label{BMO 7}
& \int_{\mathbb R^{n+1}_+\times\mathbb R^{m+1}_+} t_1 t_2  | \nabla^{(1)} \nabla^{(2)} B(x_1,x_2,t_1,t_2) |\ |\nabla_{x_1,x_2}    \nabla^{(1)} \nabla^{(2)} G(x_1,x_2,t_1,t_2) |\\
&\quad\times |  \nabla^{(2)} F(x_1,x_2,t_1,t_2) |  dx_1dx_2dt_1dt_2 \nonumber\\
&\leq C \|b\|_{{\rm BMO}_{\mathcal{F}}(\mathbb{R}^{n}\times \mathbb{R}^{m})}\| (-\Delta_{x_1,x_2})^{1\over2} g \|_{L^p(\mathbb{R}^{n}\times \mathbb{R}^{m})}\| (-\Delta_{x_2})^{1\over2} f \|_{L^{p'}(\mathbb{R}^{n}\times \mathbb{R}^{m})};\nonumber
\end{align}
\item 
\begin{align}\label{BMO 8}
& \int_{\mathbb R^{n+1}_+\times\mathbb R^{m+1}_+} t_1 t_2  | \nabla^{(1)} \nabla^{(2)} B(x_1,x_2,t_1,t_2) |\ |\nabla_{x_1,x_2}    \nabla^{(1)} \nabla^{(2)} G(x_1,x_2,t_1,t_2) |\\
&\quad\times |  F(x_1,x_2,t_1,t_2) |  dx_1dx_2dt_1dt_2 \nonumber\\
&\leq C \|b\|_{{\rm BMO}_{\mathcal{F}}(\mathbb{R}^{n}\times \mathbb{R}^{m})}\| (-\Delta_{x_1,x_2})^{1\over2}g \|_{L^p(\mathbb{R}^{n}\times \mathbb{R}^{m})}\|  f \|_{L^{p'}(\mathbb{R}^{n}\times \mathbb{R}^{m})};\nonumber
\end{align}

\item \begin{align}\label{BMO 9}
& \int_{\mathbb R^{n+1}_+\times\mathbb R^{m+1}_+} t_1 t_2  | \nabla^{(1)} \nabla^{(2)} B(x_1,x_2,t_1,t_2) |\ |\nabla_{x_2}    \nabla^{(1)} \nabla^{(2)} G(x_1,x_2,t_1,t_2) |\\
&\quad\times | \nabla^{(1)} \nabla^{(2)} F(x_1,x_2,t_1,t_2) |  dx_1dx_2dt_1dt_2 \nonumber\\
&\leq C \|b\|_{{\rm BMO}_{\mathcal{F}}(\mathbb{R}^{n}\times \mathbb{R}^{m})}\| (-\Delta_{x_2})^{1\over2}g \|_{L^p(\mathbb{R}^{n}\times \mathbb{R}^{m})}\| (-\Delta_{x_1,x_2})^{1\over2}(-\Delta_{x_2})^{1\over2} f \|_{L^{p'}(\mathbb{R}^{n}\times \mathbb{R}^{m})};\nonumber
\end{align}
\item \begin{align}\label{BMO 10}
& \int_{\mathbb R^{n+1}_+\times\mathbb R^{m+1}_+} t_1 t_2  | \nabla^{(1)} \nabla^{(2)} B(x_1,x_2,t_1,t_2) |\ |\nabla_{x_2}    \nabla^{(1)} \nabla^{(2)} G(x_1,x_2,t_1,t_2) |\\
&\quad\times | \nabla^{(1)}  F(x_1,x_2,t_1,t_2) |  dx_1dx_2dt_1dt_2 \nonumber\\
&\leq C \|b\|_{{\rm BMO}_{\mathcal{F}}(\mathbb{R}^{n}\times \mathbb{R}^{m})}\| (-\Delta_{x_2})^{1\over2} g \|_{L^p(\mathbb{R}^{n}\times \mathbb{R}^{m})}\| (-\Delta_{x_1,x_2})^{1\over2} f \|_{L^{p'}(\mathbb{R}^{n}\times \mathbb{R}^{m})};\nonumber
\end{align}
\item \begin{align}\label{BMO 11}
& \int_{\mathbb R^{n+1}_+\times\mathbb R^{m+1}_+} t_1 t_2  | \nabla^{(1)} \nabla^{(2)} B(x_1,x_2,t_1,t_2) |\ |\nabla_{x_2}    \nabla^{(1)} \nabla^{(2)} G(x_1,x_2,t_1,t_2) |\\
&\quad\times |  \nabla^{(2)} F(x_1,x_2,t_1,t_2) |  dx_1dx_2dt_1dt_2 \nonumber\\
&\leq C \|b\|_{{\rm BMO}_{\mathcal{F}}(\mathbb{R}^{n}\times \mathbb{R}^{m})}\| (-\Delta_{x_2})^{1\over2} g \|_{L^p(\mathbb{R}^{n}\times \mathbb{R}^{m})}\| (-\Delta_{x_2})^{1\over2} f \|_{L^{p'}(\mathbb{R}^{n}\times \mathbb{R}^{m})};\nonumber
\end{align}
\item 
\begin{align}\label{BMO 12}
& \int_{\mathbb R^{n+1}_+\times\mathbb R^{m+1}_+} t_1 t_2  | \nabla^{(1)} \nabla^{(2)} B(x_1,x_2,t_1,t_2) |\ |\nabla_{x_2}    \nabla^{(1)} \nabla^{(2)} G(x_1,x_2,t_1,t_2) |\\
&\quad\times |  F(x_1,x_2,t_1,t_2) |  dx_1dx_2dt_1dt_2 \nonumber\\
&\leq C \|b\|_{{\rm BMO}_{\mathcal{F}}(\mathbb{R}^{n}\times \mathbb{R}^{m})}\| (-\Delta_{x_2})^{1\over2}g \|_{L^p(\mathbb{R}^{n}\times \mathbb{R}^{m})}\|  f \|_{L^{p'}(\mathbb{R}^{n}\times \mathbb{R}^{m})};\nonumber
\end{align}

\item \begin{align}\label{BMO 13}
& \int_{\mathbb R^{n+1}_+\times\mathbb R^{m+1}_+} t_1 t_2  | \nabla^{(1)} \nabla^{(2)} B(x_1,x_2,t_1,t_2) |\ |    \nabla^{(1)} \nabla^{(2)} G(x_1,x_2,t_1,t_2) |\\
&\quad\times | \nabla^{(1)} \nabla^{(2)} F(x_1,x_2,t_1,t_2) |  dx_1dx_2dt_1dt_2 \nonumber\\
&\leq C \|b\|_{{\rm BMO}_{\mathcal{F}}(\mathbb{R}^{n}\times \mathbb{R}^{m})}\| g \|_{L^p(\mathbb{R}^{n}\times \mathbb{R}^{m})}\| (-\Delta_{x_1,x_2})^{1\over2}(-\Delta_{x_2})^{1\over2} f \|_{L^{p'}(\mathbb{R}^{n}\times \mathbb{R}^{m})};\nonumber
\end{align}
\item \begin{align}\label{BMO 14}
& \int_{\mathbb R^{n+1}_+\times\mathbb R^{m+1}_+} t_1 t_2  | \nabla^{(1)} \nabla^{(2)} B(x_1,x_2,t_1,t_2) |\ |    \nabla^{(1)} \nabla^{(2)} G(x_1,x_2,t_1,t_2) |\\
&\quad\times | \nabla^{(1)}  F(x_1,x_2,t_1,t_2) |  dx_1dx_2dt_1dt_2 \nonumber\\
&\leq C \|b\|_{{\rm BMO}_{\mathcal{F}}(\mathbb{R}^{n}\times \mathbb{R}^{m})}\| g \|_{L^p(\mathbb{R}^{n}\times \mathbb{R}^{m})}\| (-\Delta_{x_1,x_2})^{1\over2} f \|_{L^{p'}(\mathbb{R}^{n}\times \mathbb{R}^{m})};\nonumber
\end{align}
\item \begin{align}\label{BMO 15}
& \int_{\mathbb R^{n+1}_+\times\mathbb R^{m+1}_+} t_1 t_2  | \nabla^{(1)} \nabla^{(2)} B(x_1,x_2,t_1,t_2) |\ |    \nabla^{(1)} \nabla^{(2)} G(x_1,x_2,t_1,t_2) |\\
&\quad\times |  \nabla^{(2)} F(x_1,x_2,t_1,t_2) |  dx_1dx_2dt_1dt_2 \nonumber\\
&\leq C \|b\|_{{\rm BMO}_{\mathcal{F}}(\mathbb{R}^{n}\times \mathbb{R}^{m})}\|  g \|_{L^p(\mathbb{R}^{n}\times \mathbb{R}^{m})}\| (-\Delta_{x_2})^{1\over2} f \|_{L^{p'}(\mathbb{R}^{n}\times \mathbb{R}^{m})};\nonumber
\end{align}
\item 
\begin{align}\label{BMO 16}
& \int_{\mathbb R^{n+1}_+\times\mathbb R^{m+1}_+} t_1 t_2  | \nabla^{(1)} \nabla^{(2)} B(x_1,x_2,t_1,t_2) |\ |    \nabla^{(1)} \nabla^{(2)} G(x_1,x_2,t_1,t_2) |\\
&\quad\times |  F(x_1,x_2,t_1,t_2) |  dx_1dx_2dt_1dt_2 \nonumber\\
&\leq C \|b\|_{{\rm BMO}_{\mathcal{F}}(\mathbb{R}^{n}\times \mathbb{R}^{m})}\| g \|_{L^p(\mathbb{R}^{n}\times \mathbb{R}^{m})}\|  f \|_{L^{p'}(\mathbb{R}^{n}\times \mathbb{R}^{m})};\nonumber
\end{align}
\end{enumerate}

\end{prop}
\begin{proof}
To begin with, we first point out that for $f\in C^\infty_{0}(\R^{n+m})$, $F(x_1,x_2,t_1,t_2)=P_{t_1,t_2}*f(x_1,x_2)$
\begin{align*} 
 &\sup_{(y_1,y_2,t_1,t_2):\  \chi_{t_1,t_2}(x_1-y_1,x_2-y_2)\not=0} |F(y_1,y_2,t_1,t_2)|  \\
 &\leq \sup_{(y_1,y_2,t_1,t_2):\  |x_1-y_1|<t_1+t_2, |x_2-y_2|<t_2} |P_{t_1,t_2}*f(y_1,y_2)|\\
 &\leq M_1(M_2(f(\cdot_1, \cdot))(\cdot_2))(x_1,x_2),
\end{align*}
where $M_1$ and $M_2$ are the Hardy-Littlewood maximal functions on $\R^{n+m}$ and $\R^m$, respectively.

Next, based on the estimate above and from the property of the Poisson semigroup, we have
\begin{align*} 
 &\sup_{(y_1,y_2,t_1,t_2):\  \chi_{t_1,t_2}(x_1-y_1,x_2-y_2)\not=0} |\partial_{t_1}\partial_{t_2}F(y_1,y_2,t_1,t_2)|  \\
 &\leq \sup_{(y_1,y_2,t_1,t_2):\  |x_1-y_1|<t_1+t_2, |x_2-y_2|<t_2} \Big|P_{t_1,t_2}*\Big(  (-\Delta_{(1)})^{1\over2}(-\Delta_{(2)})^{1\over2} f\Big)(y_1,y_2)\Big|\\
 &\leq M_1\Bigg(M_2\bigg(\Big(  (-\Delta_{x_1,x_2})^{1\over2}(-\Delta_{x_2})^{1\over2} f\Big)(\cdot_1, \cdot)\bigg)(\cdot_2)\Bigg)(x_1,x_2).
\end{align*}
Also, we have 
\begin{align*} 
 &\sup_{(y_1,y_2,t_1,t_2):\  \chi_{t_1,t_2}(x_1-y_1,x_2-y_2)\not=0} |\nabla_{y_1,y_2}\nabla_{y_2}F(y_1,y_2,t_1,t_2)|  \\
 &\leq \sup_{(y_1,y_2,t_1,t_2):\  |x_1-y_1|<t_1+t_2, |x_2-y_2|<t_2} \Big|P_{t_1,t_2}* \Big(\nabla_{\cdot_1,\cdot_2}\nabla_{\cdot_2}f\Big)(y_1,y_2)\Big|\\
 &\leq M_1\Bigg(M_2\bigg( \Big(\nabla_{\cdot_1,\cdot_2}\nabla_{\cdot_2}f\Big)(\cdot_1, \cdot)\bigg)(\cdot_2)\Bigg)(x_1,x_2).
\end{align*}
Then, we first consider \eqref{BMO 1}.
Based on the estimates above and Corollary \ref{cor Tent},  we have
\begin{align*}
& \int_{\mathbb R^{n+1}_+\times\mathbb R^{m+1}_+} t_1 t_2  | \nabla^{(1)} \nabla^{(2)} B(x_1,x_2,t_1,t_2) |\ |\nabla_{x_1,x_2} \nabla_{x_2}   \nabla^{(1)} \nabla^{(2)} G(x_1,x_2,t_1,t_2) |\\
&\quad\times | \nabla^{(1)} \nabla^{(2)} F(x_1,x_2,t_1,t_2) |  dx_1dx_2dt_1dt_2 \nonumber\\
&\leq C \|b\|_{{\rm BMO}_{\mathcal{F}}(\mathbb{R}^{n}\times \mathbb{R}^{m})}\int_{\R^n\times\R^m} S_{\mathcal{F},L}\big( t_1t_2 \nabla_{x_1,x_2} \nabla_{x_2}   \nabla^{(1)} \nabla^{(2)} G\big) (x_1,x_2)\\
&\quad\quad\quad\times\left(M_1\Bigg(M_2\bigg(\Big(  (-\Delta_{x_1,x_2})^{1\over2}(-\Delta_{x_2})^{1\over2} f\Big)(\cdot_1, \cdot)\bigg)(\cdot_2)\Bigg)(x_1,x_2)\right. \\
&\quad\quad\quad \quad\quad\quad\quad\quad\quad\left.+  M_1\Bigg(M_2\bigg( \Big(\nabla_{\cdot_1,\cdot_2}\nabla_{\cdot_2}f\Big)(\cdot_1, \cdot)\bigg)(\cdot_2)\Bigg)(x_1,x_2) \right)dx_1dx_2   \\
&\leq C \|b\|_{{\rm BMO}_{\mathcal{F}}(\mathbb{R}^{n}\times \mathbb{R}^{m})}\\
&\quad\times\int_{\R^n\times\R^m} S_{\mathcal{F}}\big(  \nabla_{x_1,x_2} \nabla_{x_2}  (-\Delta_{x_1,x_2})^{-{1\over2}}(-\Delta_{x_2})^{-{1\over2}}  (-\Delta_{x_1,x_2})^{1\over2}(-\Delta_{x_2})^{1\over2}G\big) (x_1,x_2)\\
&\quad\quad\times\left(M_1\Bigg(M_2\bigg(\Big(  (-\Delta_{x_1,x_2})^{1\over2}(-\Delta_{x_2})^{1\over2} f\Big)(\cdot_1, \cdot)\bigg)(\cdot_2)\Bigg)(x_1,x_2)\right. \\
&\quad\quad \quad\left.+  M_1\Bigg(M_2\bigg( \Big(\nabla_{\cdot_1,\cdot_2}\nabla_{\cdot_2} (-\Delta_{\cdot_1,\cdot_2})^{-{1\over2}}(-\Delta_{\cdot_2})^{-{1\over2}}  (-\Delta_{\cdot_1,\cdot_2})^{1\over2}(-\Delta_{\cdot_2})^{1\over2}f\Big)(\cdot_1, \cdot)\bigg)(\cdot_2)\Bigg)(x_1,x_2) \right)\\
&\hskip14cmdx_1dx_2   \\
&\leq C \|b\|_{{\rm BMO}_{\mathcal{F}}(\mathbb{R}^{n}\times \mathbb{R}^{m})}\| (-\Delta_{x_1,x_2})^{1\over2}(-\Delta_{x_2})^{1\over2} g \|_{L^p(\mathbb{R}^{n}\times \mathbb{R}^{m})}\| (-\Delta_{x_1,x_2})^{1\over2}(-\Delta_{x_2})^{1\over2} f \|_{L^{p'}(\mathbb{R}^{n}\times \mathbb{R}^{m})},\nonumber
\end{align*} 
where in the second inequality the area function $S_\mathcal{F}$ is defined as in Definition \ref{S-function-Su}, and the last inequality follows from H\"older's inequality and boundedness of the maximal functions as well as the boundedness of the flag Riesz transforms. Hence we see that \eqref{BMO 1} holds.

By using similar estimate as above, we can obtain the estimates in \eqref{BMO 2}--\eqref{BMO 16}. We omit the details here since they are straightforward.
\end{proof}

\subsection{Upper bound for iterated commutators}

\begin{thm}
For every  $b\in {\rm BMO}_{\mathcal{F}}(\R^n\times\R^m)$, $g\in C_c^\infty(\mathbb R^n\times \mathbb R^m)$ and for any $i=1,2,\ldots,m+n$, $j=1,\ldots,n$, there exits a positive constant $C$ depending only on $p,n$ and $m$ such that
\begin{align}
\big\|\big[\big[b,R^{(1)}_i\big],R^{(2)}_j\big]_2(g)\big\|_{L^p(\mathbb R^n\times \mathbb R^m)} \leq C\|b\|_{{\rm BMO}_{\mathcal{F}}(\mathbb R^n\times \mathbb R^m)} \|g\|_{L^p(\mathbb R^n\times \mathbb R^m)}.
\end{align}
\end{thm}

\begin{proof}
Recall that
\begin{align*}
\big[\big[b,R^{(1)}_i\big],R^{(2)}_j\big]_2(g)(x_1,x_2) 
&= b(x_1,x_2) R^{(1)}_i \ast R^{(2)}_j\ast_2 g(x_1,x_2) - R^{(1)}_i \ast (b\cdot R^{(2)}_j\ast_2 g) (x_1,x_2)\\
&\quad - R^{(2)}_j \ast_2 \big( b\cdot R^{(1)}_i \ast g\big) (x_1,x_2) + R^{(2)}_j\ast_2 R^{(1)}_i\ast (b\cdot g) (x_1,x_2).
\end{align*}
Hence, for every $f\in C_c^\infty(\mathbb R^n\times \mathbb R^m)$, we have
\begin{align*}
\big\langle f, \big[\big[b,R^{(1)}_i\big],R^{(2)}_j\big]_2(g) \big\rangle
&=\big\langle f \cdot b, R^{(1)}_i \ast R^{(2)}_j\ast_2 g   \big\rangle+  \big\langle R^{(1)}_i \ast f, b\cdot R^{(2)}_j\ast_2 g  \big\rangle\\
&\quad + \big\langle R^{(2)}_j \ast_2 f,   b\cdot R^{(1)}_i \ast g  \big\rangle+  \big\langle R^{(2)}_j\ast_2 R^{(1)}_i\ast f, b\cdot g  \big\rangle.
\end{align*}

Denote by $B, F, G$ the flag harmonic extension of the functions $b,f,g$, respectively, as defined in \eqref{poisson integral}. And for each fixed $i,j$, denote by
$ \widetilde{R^{(1)}_i \ast f} $,  $\widetilde{R^{(2)}_j\ast_2 f} $ and $ \widetilde{R^{(1)}_i \ast R^{(2)}_j\ast_2 f}$  the flag harmonic extension of 
$R^{(1)}_i \ast f$,  $R^{(2)}_j\ast_2 f$ and $R^{(1)}_i \ast R^{(2)}_j\ast_2 f$.

Then we write
\begin{align}\label{inner product}
&\big\langle f, \big[\big[b,R^{(1)}_i\big],R^{(2)}_j\big]_2(g) \big\rangle\\
&=  \int_{\mathbb R^{n+1}_+\times\mathbb R^{m+1}_+} t_1 \partial_{t_1}^2t_2 \partial_{t_2}^2 \bigg(
  F \cdot B \cdot \widetilde{R^{(1)}_i \ast R^{(2)}_j\ast_2 g}   +  \widetilde{ R^{(1)}_i \ast f}\, \cdot B\cdot \widetilde{R^{(2)}_j\ast_2 g} \nonumber \\
&\quad +  \widetilde{R^{(2)}_j \ast_2 f}\, \cdot  B\cdot \widetilde{R^{(1)}_i \ast g} + \widetilde{ R^{(2)}_j\ast_2 R^{(1)}_i\ast f}\, \cdot B\cdot G\bigg)   dx_1dx_2dt_1dt_2. \nonumber
\end{align}
We now claim that the right-hand side of \eqref{inner product}
is bounded by
\begin{align}\label{inner product bound}
C\|b\|_{{\rm BMO}_{\mathcal{F}}(\mathbb R^n\times \mathbb R^m)} \|g\|_{L^p(\mathbb R^n\times \mathbb R^m)}
\|f\|_{L^{p'}(\mathbb R^n\times \mathbb R^m)}.
\end{align}

To see this, we compute the derivatives $t_1 \partial_{t_1}^2t_2 \partial_{t_2}^2$
for the integrand in the right-hand side of \eqref{inner product}. Then we have the following terms:
\begin{align}\label{C1}
\mathcal {C}_1=& \int_{\mathbb R^{n+1}_+\times\mathbb R^{m+1}_+}\bigg(
  t_1 \partial_{t_1}^2t_2 \partial_{t_2}^2 B \cdot F  \cdot \widetilde{R^{(1)}_i \ast R^{(2)}_j\ast_2 g}   + t_1 \partial_{t_1}^2t_2 \partial_{t_2}^2 B\cdot \widetilde{ R^{(1)}_i \ast f}\, \cdot \widetilde{R^{(2)}_j\ast_2 g} \nonumber \\
&\ + t_1 \partial_{t_1}^2t_2 \partial_{t_2}^2 B\cdot  \widetilde{R^{(2)}_j \ast_2 f}\, \cdot \widetilde{R^{(1)}_i \ast g} +t_1 \partial_{t_1}^2t_2 \partial_{t_2}^2 B\cdot  \widetilde{ R^{(2)}_j\ast_2 R^{(1)}_i\ast f}\, \cdot G\bigg) dx_1dx_2dt_1dt_2;
\end{align}
\begin{align}\label{C2}
\mathcal {C}_2=&\int_{\mathbb R^{n+1}_+\times\mathbb R^{m+1}_+}
  t_1 \partial_{t_1}^2t_2 \partial_{t_2} B \cdot  \partial_{t_2}\Big( F  \cdot \widetilde{R^{(1)}_i \ast R^{(2)}_j\ast_2 g} \Big)  + t_1 \partial_{t_1}^2t_2 \partial_{t_2} B\cdot  \partial_{t_2}\Big( \widetilde{ R^{(1)}_i \ast f}\, \cdot \widetilde{R^{(2)}_j\ast_2 g}\Big) \nonumber \\
&\ + t_1 \partial_{t_1}^2t_2 \partial_{t_2} B\cdot   \partial_{t_2}\Big( \widetilde{R^{(2)}_j \ast_2 f}\, \cdot \widetilde{R^{(1)}_i \ast g}\Big)\nonumber\\
&\quad+t_1 \partial_{t_1}^2t_2 \partial_{t_2} B\cdot  \partial_{t_2}\Big( \widetilde{ R^{(2)}_j\ast_2 R^{(1)}_i\ast f}\, \cdot G\Big) dx_1dx_2dt_1dt_2;
\end{align}
\begin{align}\label{C3}
\mathcal {C}_3=&\int_{\mathbb R^{n+1}_+\times\mathbb R^{m+1}_+}
  t_1 \partial_{t_1}t_2 \partial_{t_2}^2 B \cdot  \partial_{t_1}\Big( F  \cdot \widetilde{R^{(1)}_i \ast R^{(2)}_j\ast_2 g} \Big)  + t_1 \partial_{t_1}t_2 \partial_{t_2}^2 B\cdot  \partial_{t_1}\Big( \widetilde{ R^{(1)}_i \ast f}\, \cdot \widetilde{R^{(2)}_j\ast_2 g}\Big) \nonumber \\
&\ + t_1 \partial_{t_1}t_2 \partial_{t_2}^2 B\cdot   \partial_{t_1}\Big( \widetilde{R^{(2)}_j \ast_2 f}\, \cdot \widetilde{R^{(1)}_i \ast g}\Big) \nonumber\\
&\quad+t_1 \partial_{t_1}t_2 \partial_{t_2}^2 B\cdot  \partial_{t_1}\Big( \widetilde{ R^{(2)}_j\ast_2 R^{(1)}_i\ast f}\, \cdot G\Big) dx_1dx_2dt_1dt_2;
\end{align}
\begin{align}\label{C4}
\mathcal {C}_4=&\int_{\mathbb R^{n+1}_+\times\mathbb R^{m+1}_+}
  t_1 \partial_{t_1}t_2 \partial_{t_2} B \cdot  \partial_{t_1}\partial_{t_2}\Big( F  \cdot \widetilde{R^{(1)}_i \ast R^{(2)}_j\ast_2 g} \Big)  \nonumber\\
&\quad+ t_1 \partial_{t_1}t_2 \partial_{t_2} B\cdot  \partial_{t_1}\partial_{t_2}\Big( \widetilde{ R^{(1)}_i \ast f}\, \cdot \widetilde{R^{(2)}_j\ast_2 g}\Big)  + t_1 \partial_{t_1}t_2 \partial_{t_2} B\cdot   \partial_{t_1}\partial_{t_2}\Big( \widetilde{R^{(2)}_j \ast_2 f}\, \cdot \widetilde{R^{(1)}_i \ast g}\Big) \nonumber\\
&\quad+t_1 \partial_{t_1}t_2 \partial_{t_2} B\cdot  \partial_{t_1}\partial_{t_2}\Big( \widetilde{ R^{(2)}_j\ast_2 R^{(1)}_i\ast f}\, \cdot G\Big) dx_1dx_2dt_1dt_2;
\end{align}
\begin{align}\label{C5}
\mathcal {C}_5=&\int_{\mathbb R^{n+1}_+\times\mathbb R^{m+1}_+}
  t_1 \partial_{t_1}t_2 B \cdot  \partial_{t_1}\partial_{t_2}^2\Big( F  \cdot \widetilde{R^{(1)}_i \ast R^{(2)}_j\ast_2 g} \Big)  + t_1 \partial_{t_1}t_2  B\cdot  \partial_{t_1}\partial_{t_2}^2\Big( \widetilde{ R^{(1)}_i \ast f}\, \cdot \widetilde{R^{(2)}_j\ast_2 g}\Big) \nonumber \\
&\ + t_1 \partial_{t_1}t_2  B\cdot   \partial_{t_1}\partial_{t_2}^2\Big( \widetilde{R^{(2)}_j \ast_2 f}\, \cdot \widetilde{R^{(1)}_i \ast g}\Big) \nonumber\\
&\quad+t_1 \partial_{t_1}t_2  B\cdot  \partial_{t_1}\partial_{t_2}^2\Big( \widetilde{ R^{(2)}_j\ast_2 R^{(1)}_i\ast f}\, \cdot G\Big) dx_1dx_2dt_1dt_2;
\end{align}
\begin{align}\label{C6}
\mathcal {C}_6=&\int_{\mathbb R^{n+1}_+\times\mathbb R^{m+1}_+}
  t_1 t_2 \partial_{t_2} B \cdot  \partial_{t_1}^2\partial_{t_2}\Big( F  \cdot \widetilde{R^{(1)}_i \ast R^{(2)}_j\ast_2 g} \Big)  + t_1t_2 \partial_{t_2} B\cdot  \partial_{t_1}^2\partial_{t_2}\Big( \widetilde{ R^{(1)}_i \ast f}\, \cdot \widetilde{R^{(2)}_j\ast_2 g}\Big) \nonumber \\
&\ + t_1 t_2 \partial_{t_2} B\cdot   \partial_{t_1}^2\partial_{t_2}\Big( \widetilde{R^{(2)}_j \ast_2 f}\, \cdot \widetilde{R^{(1)}_i \ast g}\Big) \nonumber\\
&\quad+t_1 t_2 \partial_{t_2} B\cdot  \partial_{t_1}^2\partial_{t_2}\Big( \widetilde{ R^{(2)}_j\ast_2 R^{(1)}_i\ast f}\, \cdot G\Big) dx_1dx_2dt_1dt_2;
\end{align}
\begin{align}\label{C7}
\mathcal {C}_7=&\int_{\mathbb R^{n+1}_+\times\mathbb R^{m+1}_+}
  t_1 t_2 \partial_{t_2}^2 B \cdot  \partial_{t_1}^2\Big( F  \cdot \widetilde{R^{(1)}_i \ast R^{(2)}_j\ast_2 g} \Big)  + t_1t_2 \partial_{t_2}^2 B\cdot  \partial_{t_1}^2\Big( \widetilde{ R^{(1)}_i \ast f}\, \cdot \widetilde{R^{(2)}_j\ast_2 g}\Big) \nonumber \\
&\ + t_1 t_2 \partial_{t_2}^2 B\cdot   \partial_{t_1}^2\Big( \widetilde{R^{(2)}_j \ast_2 f}\, \cdot \widetilde{R^{(1)}_i \ast g}\Big) +t_1 t_2 \partial_{t_2}^2 B\cdot  \partial_{t_1}^2\Big( \widetilde{ R^{(2)}_j\ast_2 R^{(1)}_i\ast f}\, \cdot G\Big) dx_1dx_2dt_1dt_2;
\end{align}
\begin{align}\label{C8}
\mathcal {C}_8=&\int_{\mathbb R^{n+1}_+\times\mathbb R^{m+1}_+}
  t_1 t_2 \partial_{t_1}^2 B \cdot  \partial_{t_2}^2\Big( F  \cdot \widetilde{R^{(1)}_i \ast R^{(2)}_j\ast_2 g} \Big)  + t_1t_2 \partial_{t_1}^2 B\cdot  \partial_{t_2}^2\Big( \widetilde{ R^{(1)}_i \ast f}\, \cdot \widetilde{R^{(2)}_j\ast_2 g}\Big) \nonumber \\
&\ + t_1 t_2 \partial_{t_1}^2 B\cdot   \partial_{t_2}^2\Big( \widetilde{R^{(2)}_j \ast_2 f}\, \cdot \widetilde{R^{(1)}_i \ast g}\Big) +t_1 t_2 \partial_{t_1}^2 B\cdot  \partial_{t_2}^2\Big( \widetilde{ R^{(2)}_j\ast_2 R^{(1)}_i\ast f}\, \cdot G\Big) dx_1dx_2dt_1dt_2;
\end{align}
\begin{align}\label{C9}
\mathcal {C}_9=&\int_{\mathbb R^{n+1}_+\times\mathbb R^{m+1}_+}
  t_1 t_2 B \cdot  \partial_{t_1}^2 \partial_{t_2}^2\Big( F  \cdot \widetilde{R^{(1)}_i \ast R^{(2)}_j\ast_2 g} \Big)  + t_1t_2  B\cdot  \partial_{t_1}^2\partial_{t_2}^2\Big( \widetilde{ R^{(1)}_i \ast f}\, \cdot \widetilde{R^{(2)}_j\ast_2 g}\Big) \nonumber \\
&\ + t_1 t_2  B\cdot   \partial_{t_1}^2\partial_{t_2}^2\Big( \widetilde{R^{(2)}_j \ast_2 f}\, \cdot \widetilde{R^{(1)}_i \ast g}\Big) +t_1 t_2  B\cdot  \partial_{t_1}^2\partial_{t_2}^2\Big( \widetilde{ R^{(2)}_j\ast_2 R^{(1)}_i\ast f}\, \cdot G\Big)dx_1dx_2dt_1dt_2.
\end{align}

We first consider $\mathcal C_1$. Note that
$  \partial_{t_2}^2 B = -\Delta_{x_2} B = -\nabla_{x_2}\cdot \nabla_{x_2} B $ and that
$  \partial_{t_1}^2 B = -\Delta_{x_1,x_2} B = -\nabla_{x_1,x_2}\cdot \nabla_{x_1,x_2} B $.
So, integration by parts gives
\begin{align*}
 \mathcal {C}_1
&=\int_{\mathbb R^{n+1}_+\times\mathbb R^{m+1}_+}
  t_1 t_2 \nabla_{x_1,x_2} \nabla_{x_2} B \cdot \nabla_{x_1,x_2} \nabla_{x_2} \Big(F  \cdot \widetilde{R^{(1)}_i \ast R^{(2)}_j\ast_2 g} \Big) \\
  & \quad + t_1 t_2 \nabla_{x_1,x_2} \nabla_{x_2} B\cdot \nabla_{x_1,x_2} \nabla_{x_2} \Big(\widetilde{ R^{(1)}_i \ast f}\, \cdot \widetilde{R^{(2)}_j\ast_2 g}\Big) \nonumber \\
&\quad + t_1 t_2 \nabla_{x_1,x_2} \nabla_{x_2} B\cdot \nabla_{x_1,x_2} \nabla_{x_2}\Big( \widetilde{R^{(2)}_j \ast_2 f}\, \cdot \widetilde{R^{(1)}_i \ast g}\Big)\\
&\quad +t_1 t_2 \nabla_{x_1,x_2} \nabla_{x_2} B\cdot \nabla_{x_1,x_2} \nabla_{x_2}\Big( \widetilde{ R^{(2)}_j\ast_2 R^{(1)}_i\ast f}\, \cdot G\Big) \ dx_1dx_2dt_1dt_2\\
&=: \mathcal {C}_{1,1}+\mathcal {C}_{1,2}+\mathcal {C}_{1,3}+\mathcal {C}_{1,4}.
\end{align*}
For the first term, it is clear that 
\begin{align*}
 \mathcal {C}_{1,1}
&=\int_{\mathbb R^{n+1}_+\times\mathbb R^{m+1}_+}
  t_1 t_2 \nabla_{x_1,x_2} \nabla_{x_2} B \cdot \nabla_{x_1,x_2} \nabla_{x_2} F  \cdot \widetilde{R^{(1)}_i \ast R^{(2)}_j\ast_2 g}  
 \ dx_1dx_2dt_1dt_2\\
&=\int_{\mathbb R^{n+1}_+\times\mathbb R^{m+1}_+}
  t_1 t_2 \nabla_{x_1,x_2} \nabla_{x_2} B \cdot \nabla_{x_1,x_2}  F  \cdot \nabla_{x_2} \widetilde{R^{(1)}_i \ast R^{(2)}_j\ast_2 g}  
 \ dx_1dx_2dt_1dt_2\\
&=\int_{\mathbb R^{n+1}_+\times\mathbb R^{m+1}_+}
  t_1 t_2 \nabla_{x_1,x_2} \nabla_{x_2} B \cdot  \nabla_{x_2} F  \cdot  \nabla_{x_1,x_2}\widetilde{R^{(1)}_i \ast R^{(2)}_j\ast_2 g} 
 \ dx_1dx_2dt_1dt_2\\
&=\int_{\mathbb R^{n+1}_+\times\mathbb R^{m+1}_+}
  t_1 t_2 \nabla_{x_1,x_2} \nabla_{x_2} B \cdot  F  \cdot \nabla_{x_1,x_2} \nabla_{x_2} \widetilde{R^{(1)}_i \ast R^{(2)}_j\ast_2 g}  
 \ dx_1dx_2dt_1dt_2\\
&=: \mathcal {C}_{1,1,1}+\mathcal {C}_{1,1,2}+\mathcal {C}_{1,1,3}+\mathcal {C}_{1,1,4}.
\end{align*}
It is direct that $\mathcal {C}_{1,1,1}$ and $\mathcal {C}_{1,1,4}$ can be handled by using 
\eqref{BMO 2}, and $\mathcal {C}_{1,1,2}$ and $\mathcal {C}_{1,1,3}$ can be handled by using 
\eqref{BMO 3}, which gives that
$$ \mathcal {C}_{1,1} \leq C\|b\|_{{\rm BMO}_{\mathcal{F}}(\mathbb R^n\times \mathbb R^m)} \|g\|_{L^p(\mathbb R^n\times \mathbb R^m)}
\|f\|_{L^{p'}(\mathbb R^n\times \mathbb R^m)}.
 $$


Symmetrically we obtain the estimate for $\mathcal {C}_{1,4}$, and using similar estimates we can handle $\mathcal {C}_{1,2}$ and 
 $\mathcal {C}_{1,3}$. All these three terms are have the same upper as $\mathcal {C}_{1,1}$ above.

Next, for $\mathcal C_2$, note that
$  \partial_{t_1}^2 B = -\Delta_{x_1,x_2} B = -\nabla_{x_1,x_2}\cdot \nabla_{x_1,x_2} B $.
Thus, similar to the term $\mathcal C_1$,
 by integration by parts, we have
\begin{align*}
 \mathcal {C}_2
&=-\int_{\mathbb R^{n+1}_+\times\mathbb R^{m+1}_+}
  t_1 t_2 \nabla_{x_1,x_2} \partial_{t_2} B \cdot \nabla_{x_1,x_2} \partial_{t_2}  \Big(F  \cdot \widetilde{R^{(1)}_i \ast R^{(2)}_j\ast_2 g} \Big) \\
  & \quad + t_1 t_2 \nabla_{x_1,x_2} \partial_{t_2}  B\cdot \nabla_{x_1,x_2} \partial_{t_2}  \Big(\widetilde{ R^{(1)}_i \ast f}\, \cdot \widetilde{R^{(2)}_j\ast_2 g}\Big) \nonumber \\
&\quad + t_1 t_2 \nabla_{x_1,x_2} \partial_{t_2}  B\cdot \nabla_{x_1,x_2} \partial_{t_2} \Big( \widetilde{R^{(2)}_j \ast_2 f}\, \cdot \widetilde{R^{(1)}_i \ast g}\Big)\\
&\quad +t_1 t_2 \nabla_{x_1,x_2} \partial_{t_2}  B\cdot \nabla_{x_1,x_2} \partial_{t_2} \Big( \widetilde{ R^{(2)}_j\ast_2 R^{(1)}_i\ast f}\, \cdot G\Big) \ dx_1dx_2dt_1dt_2\\
&=: \mathcal {C}_{2,1}+\mathcal {C}_{2,2}+\mathcal {C}_{2,3}+\mathcal {C}_{2,4}.
\end{align*}
Again, the upper bounds from the four terms above can be obtained  by applying  Proposition 
\ref{prop BMO estimate}, and they are all controlled by 
$$C\|b\|_{{\rm BMO}_{\mathcal{F}}(\mathbb R^n\times \mathbb R^m)} \|g\|_{L^p(\mathbb R^n\times \mathbb R^m)}
\|f\|_{L^{p'}(\mathbb R^n\times \mathbb R^m)}.$$
The term $\mathcal C_3$ can be handled symmetrically to $\mathcal C_2$ and we obtain the same upper bounds.

For the term $\mathcal C_4$, by noting that $|\partial_{t_1} \partial_{t_2} B(x_1,x_2,t_1,t_2)|$ is bounded by
$| \nabla^{(1)} \nabla^{(2)} B(x_1,x_2,t_1,t_2) |$, we obtain that 
$\mathcal C_4$ is bounded by 
 $$C\|b\|_{{\rm BMO}_{\mathcal{F}}(\mathbb R^n\times \mathbb R^m)} \|g\|_{L^p(\mathbb R^n\times \mathbb R^m)}
\|f\|_{L^{p'}(\mathbb R^n\times \mathbb R^m)},$$
where we apply again the upper bounds in Proposition \ref{prop BMO estimate}.


We now turn to the  term $\mathcal C_9$. We first point out the following equalities:
\begin{align*}
&\partial_{t_1} \widetilde{R^{(1)}_i \ast R^{(2)}_j\ast_2 g}(x_1,x_2) = -c \partial_{(x_{1},x_2),i}\ \widetilde{R^{(2)}_j\ast_2 g}(x_1,x_2)\\
&\partial_{t_1}^2 \widetilde{R^{(1)}_i \ast R^{(2)}_j\ast_2 g}(x_1,x_2) = -c\partial_{t_1} \partial_{(x_{1},x_2),i}\ \widetilde{R^{(2)}_j\ast_2 g}(x_1,x_2)\\
&\partial_{t_2} \widetilde{R^{(1)}_i \ast R^{(2)}_j\ast_2 g}(x_1,x_2) = -c \partial_{x_{2,j}}\ \widetilde{R^{(1)}_i\ast g}(x_1,x_2)\\
&\partial_{t_2}^2 \widetilde{R^{(1)}_i \ast R^{(2)}_j\ast_2 g}(x_1,x_2) = -c\partial_{t_2}  \partial_{x_{2,j}}\ \widetilde{R^{(1)}_i\ast g}(x_1,x_2)\\
&\partial_{t_1} \widetilde{ R^{(1)}_i \ast f}=-c \partial_{(x_{1},x_2),i} \widetilde {f},\\
&\partial_{t_1}^2 \widetilde{ R^{(1)}_i \ast f}=-c \partial_{t_1}\partial_{(x_{1},x_2),i} \widetilde {f},\\
&\partial_{t_2} \widetilde{R^{(2)}_j\ast_2 g}=-c  \partial_{x_{2,j}} \widetilde{ g},\\
&\partial_{t_2}^2 \widetilde{R^{(2)}_j\ast_2 g}=-c \partial_{t_2}  \partial_{x_{2,j}} \widetilde{ g}.
\end{align*}
Then for the term $\mathcal C_9$, we get
\begin{align*}
&\partial_{t_1}^2 \partial_{t_2}^2\Big( F  \cdot \widetilde{R^{(1)}_i \ast R^{(2)}_j\ast_2 g} +\widetilde{ R^{(1)}_i \ast f}\, \cdot \widetilde{R^{(2)}_j\ast_2 g}+ \widetilde{R^{(2)}_j \ast_2 f}\, \cdot \widetilde{R^{(1)}_i \ast g} +   \widetilde{ R^{(2)}_j\ast_2 R^{(1)}_i\ast f}\, \cdot G\Big)\\
&= 4 \partial_{(x_{1},x_2),i}\partial_{t_1}\partial_{x_{2,j}}\partial_{t_2} (FG)\\
&\quad\quad  -2 \nabla_{x_1,x_2} \partial_{x_{2,j}}\partial_{t_2} \Big( \nabla_{x_1,x_2}\widetilde{ R^{(1)}_i \ast f} \cdot G \Big)
-2 \nabla_{x_1,x_2} \partial_{x_{2,j}}\partial_{t_2} \Big( F \cdot \nabla_{x_1,x_2}\widetilde{ R^{(1)}_i \ast g}  \Big)\\
&\quad\quad +2 \nabla_{x_1,x_2} \partial_{x_{2,j}}\partial_{t_2} \Big( \nabla_{x_1,x_2} F\cdot \widetilde{ R^{(1)}_i \ast g} \Big)+
2 \nabla_{x_1,x_2} \partial_{x_{2,j}}\partial_{t_2} \Big(  \widetilde{ R^{(1)}_i \ast f} \cdot\nabla_{x_1,x_2} G\Big)  \\
&\quad-2\partial_{(x_{1},x_2),i}\partial_{t_1}\nabla_{x_2} \Big(  \nabla_{x_2}\widetilde{ R^{(2)}_j\ast f}\cdot G  \Big)\\
&\quad\quad +\nabla_{x_1,x_2} \nabla_{x_2}\Big(   \nabla_{x_1,x_2} \nabla_{x_2} \widetilde{ R^{(2)}_j\ast_2 R^{(1)}_i\ast f} \cdot G   \Big)+\nabla_{x_1,x_2} \nabla_{x_2}\Big(   \nabla_{x_2} \widetilde{ R^{(2)}_j\ast_2 f} \cdot \nabla_{x_1,x_2}    \widetilde{ R^{(1)}_i \ast g}  \Big)\\
&\quad\quad -\nabla_{x_1,x_2} \nabla_{x_2}\Big(   \nabla_{x_1,x_2} \nabla_{x_2} \widetilde{ R^{(2)}_j\ast_2  f} \cdot   \widetilde{ R^{(1)}_i \ast g}     \Big)- \nabla_{x_1,x_2} \nabla_{x_2}\Big(   \nabla_{x_2} \widetilde{ R^{(2)}_j\ast_2 R^{(1)}_i\ast f} \cdot \nabla_{x_1,x_2}    G  \Big)\\
&\quad-2\partial_{(x_{1},x_2),i}\partial_{t_1}\nabla_{x_2} \Big(  F \cdot \nabla_{x_2}\widetilde{ R^{(2)}_j\ast g}  \Big)\\
&\quad\quad +\nabla_{x_1,x_2} \nabla_{x_2}\Big(   \nabla_{x_1,x_2}  \widetilde{  R^{(1)}_i\ast f} \cdot \nabla_{x_2} \widetilde{  R^{(2)}_j\ast_2 g}   \Big)+\nabla_{x_1,x_2} \nabla_{x_2}\Big(  F \cdot \nabla_{x_1,x_2}  \nabla_{x_2}   \widetilde{ R^{(1)}_i \ast R^{(2)}_j\ast_2 g}  \Big)\\
&\quad\quad -\nabla_{x_1,x_2} \nabla_{x_2}\Big(   \nabla_{x_1,x_2} \nabla_{x_2} \widetilde{ R^{(2)}_j\ast_2  f} \cdot   \widetilde{ R^{(1)}_i \ast g}     \Big)- \nabla_{x_1,x_2} \nabla_{x_2}\Big(   \nabla_{x_2} \widetilde{ R^{(2)}_j\ast_2 R^{(1)}_i\ast f} \cdot \nabla_{x_1,x_2}    G  \Big)\\
&\quad+2\partial_{(x_{1},x_2),i}\partial_{t_1}\nabla_{x_2} \Big(  \nabla_{x_2} F \cdot \widetilde{ R^{(2)}_j\ast g}  \Big)\\
&\quad\quad -\nabla_{x_1,x_2} \nabla_{x_2}\Big(   \nabla_{x_1,x_2} \nabla_{x_2} \widetilde{  R^{(1)}_i\ast f} \cdot  \widetilde{  R^{(2)}_j\ast_2 g}   \Big)-\nabla_{x_1,x_2} \nabla_{x_2}\Big( \nabla_{x_2}   F \cdot \nabla_{x_1,x_2}   \widetilde{ R^{(1)}_i \ast R^{(2)}_j\ast_2 g}  \Big)\\
&\quad\quad +\nabla_{x_1,x_2} \nabla_{x_2}\Big(   \nabla_{x_1,x_2} \nabla_{x_2} F \cdot   \widetilde{ R^{(1)}_i \ast R^{(2)}_j\ast_2 g}     \Big)+     \nabla_{x_1,x_2} \nabla_{x_2}\Big(   \nabla_{x_2} \widetilde{  R^{(1)}_i\ast f} \cdot \nabla_{x_1,x_2}   \widetilde{  R^{(2)}_j\ast_2 g}    \Big)\\
&\quad+2\partial_{(x_{1},x_2),i}\partial_{t_1}\nabla_{x_2} \Big( \widetilde{ R^{(2)}_j\ast f} \cdot \nabla_{x_2} G    \Big)\\
&\quad\quad -\nabla_{x_1,x_2} \nabla_{x_2}\Big(   \nabla_{x_1,x_2}  \widetilde{  R^{(1)}_i\ast R^{(2)}_j\ast_2 f} \cdot  \nabla_{x_2} G   \Big)-\nabla_{x_1,x_2} \nabla_{x_2}\Big(        \widetilde{  R^{(2)}_j\ast_2 f} \cdot \nabla_{x_1,x_2} \nabla_{x_2}  \widetilde{ R^{(1)}_i \ast  g}  \Big)\\
&\quad\quad +\nabla_{x_1,x_2} \nabla_{x_2}\Big(   \nabla_{x_1,x_2}  \widetilde{  R^{(2)}_j\ast_2 f}    \cdot   \nabla_{x_2} \widetilde{ R^{(1)}_i \ast  g}     \Big)+     \nabla_{x_1,x_2} \nabla_{x_2}\Big(    \widetilde{  R^{(1)}_i\ast R^{(2)}_j\ast_2 f} \cdot \nabla_{x_1,x_2} \nabla_{x_2}   G    \Big).
\end{align*}

Thus, we input the above 25 terms back into the right-hand side of $\mathcal C_9$ and obtain the terms as follows:
\begin{align*}
\mathcal {C}_9&=\int_{\mathbb R^{n+1}_+\times\mathbb R^{m+1}_+}
  t_1 t_2 B \cdot  \partial_{t_1}^2 \partial_{t_2}^2\Big( F  \cdot \widetilde{R^{(1)}_i \ast R^{(2)}_j\ast_2 g}  + \widetilde{ R^{(1)}_i \ast f}\, \cdot \widetilde{R^{(2)}_j\ast_2 g} \\
  &\quad+ \widetilde{R^{(2)}_j \ast_2 f}\, \cdot \widetilde{R^{(1)}_i \ast g} + \widetilde{ R^{(2)}_j\ast_2 R^{(1)}_i\ast f}\, \cdot G\Big)dx_1dx_2dt_1dt_2\\
  &=4\int_{\mathbb R^{n+1}_+\times\mathbb R^{m+1}_+}  t_1 t_2 \partial_{(x_{1},x_2),i} \partial_{x_{2,j}} B \cdot\partial_{t_1}\partial_{t_2} (FG) dx_1dx_2dt_1dt_2\\
  &\quad -2\int_{\mathbb R^{n+1}_+\times\mathbb R^{m+1}_+}  t_1 t_2 \nabla_{x_1,x_2} \partial_{x_{2,j}} B \cdot   \partial_{t_2} \Big( \nabla_{x_1,x_2}\widetilde{ R^{(1)}_i \ast f} \cdot G \Big)dx_1dx_2dt_1dt_2\\
  &\cdots\\
  &\cdots\\
  &\quad +\int_{\mathbb R^{n+1}_+\times\mathbb R^{m+1}_+}  t_1 t_2 \nabla_{x_1,x_2} \nabla_{x_2} B \cdot  \Big(    \widetilde{  R^{(1)}_i\ast R^{(2)}_j\ast_2 f} \cdot \nabla_{x_1,x_2} \nabla_{x_2}   G    \Big) dx_1dx_2dt_1dt_2\\
  &=\mathcal {C}_{9,1}+\mathcal {C}_{9,2}+\cdots+\mathcal {C}_{9,25},
\end{align*}
where we get all these terms from the equality $\partial_{t_1}^2 \partial_{t_2}^2\big( \cdots\big)$
by integration by parts and taking all the gradients or partial derivatives with respect to $x_1,x_2$ to the function $B$. 
By applying  Proposition 
\ref{prop BMO estimate} to all these terms, we obtain that  they are all controlled by 
$$C\|b\|_{{\rm BMO}_{\mathcal{F}}(\mathbb R^n\times \mathbb R^m)} \|g\|_{L^p(\mathbb R^n\times \mathbb R^m)}
\|f\|_{L^{p'}(\mathbb R^n\times \mathbb R^m)}.$$

Next we consider the term $\mathcal C_5$, which can be consider as a cross term in between $\mathcal C_1$ and $\mathcal C_9$.  To continue, we write
\begin{align*}
&\partial_{t_2}^2\Big( F  \cdot \widetilde{R^{(1)}_i \ast R^{(2)}_j\ast_2 g} + \widetilde{ R^{(1)}_i \ast f}\, \cdot \widetilde{R^{(2)}_j\ast_2 g} + \widetilde{R^{(2)}_j \ast_2 f}\, \cdot \widetilde{R^{(1)}_i \ast g} + \widetilde{ R^{(2)}_j\ast_2 R^{(1)}_i\ast f}\, \cdot G\Big)\\
&=\partial_{t_2}^2\Big( F  \cdot \widetilde{R^{(2)}_j\ast_2 (R^{(1)}_i \ast  g)}  + \widetilde{R^{(2)}_j \ast_2 f}\, \cdot \widetilde{R^{(1)}_i \ast g} \Big)\\
&\quad+\partial_{t_2}^2\Big(  \widetilde{ R^{(1)}_i \ast f}\, \cdot \widetilde{R^{(2)}_j\ast_2 g}  + \widetilde{ R^{(2)}_j\ast_2 (R^{(1)}_i\ast f)}\, \cdot G\Big)\\
&= E_1+E_2.
\end{align*}
For the term $E_1$, we write
\begin{align*}
E_1
&=-2 \partial_{x_{2,j}}\partial_{t_2} \Big( F\cdot  \widetilde {R^{(1)}_i \ast  g} \Big)+\nabla_{x_2}\Big( \nabla_{x_2} \widetilde{R^{(2)}_j \ast_2 f}  \cdot \widetilde {R^{(1)}_i \ast  g} + F  \cdot  \nabla_{x_2}  \widetilde {R^{(2)}_j \ast_2 R^{(1)}_i \ast  g} \\
&\quad -\nabla_{x_2} F \cdot \widetilde{R^{(2)}_j \ast_2 R^{(1)}_i \ast  g}   - \widetilde{R^{(2)}_j \ast_2 f} \cdot  \nabla_{x_2}  \widetilde {R^{(1)}_i \ast  g} \Big).  
\end{align*}
For the term $E_2$, we write
\begin{align*}
E_2
&=-2 \partial_{x_{2,j}}\partial_{t_2} \Big( \widetilde{ R^{(1)}_i \ast f}\cdot G \Big)+  \nabla_{x_2}\Big( \nabla_{x_2} \widetilde{R^{(2)}_j \ast_2R^{(1)}_i \ast f}  \cdot G + \widetilde{ R^{(1)}_i \ast f}  \cdot  \nabla_{x_2}  \widetilde {R^{(2)}_j \ast_2   g} \\
&\quad -\nabla_{x_2} \widetilde{ R^{(1)}_i \ast f} \cdot \widetilde{R^{(2)}_j \ast_2 g}   - \widetilde{R^{(2)}_j \ast_2  f} \cdot  \nabla_{x_2} G \Big).  
\end{align*}
As a consequence, by substituting the above 10 terms in the right-hand side of the equalities $E_1$ and $E_2$ back in to the term $\mathcal C_5$,  we have that 
\begin{align*}
\mathcal {C}_5&=2\int_{\mathbb R^{n+1}_+\times\mathbb R^{m+1}_+}  t_1 \partial_{t_1}t_2\partial_{x_{2,j}}  B \cdot  \partial_{t_1}  \partial_{t_2} \Big( F\cdot  \widetilde {R^{(1)}_i \ast  g} \Big)
     dx_1dx_2dt_1dt_2\\
    &\quad- \int_{\mathbb R^{n+1}_+\times\mathbb R^{m+1}_+}  t_1 \partial_{t_1}t_2\nabla_{x_2} B \cdot  \partial_{t_1}  \Big( \nabla_{x_2} \widetilde{R^{(2)}_j \ast_2 f}  \cdot \widetilde {R^{(1)}_i \ast  g}\Big)
     dx_1dx_2dt_1dt_2\\
    &\quad- \int_{\mathbb R^{n+1}_+\times\mathbb R^{m+1}_+}  t_1 \partial_{t_1}t_2 \nabla_{x_2}B \cdot  \partial_{t_1}  \Big(  F  \cdot  \nabla_{x_2}  \widetilde {R^{(2)}_j \ast_2 R^{(1)}_i \ast  g}\Big)
     dx_1dx_2dt_1dt_2\\
    &\quad+ \int_{\mathbb R^{n+1}_+\times\mathbb R^{m+1}_+}  t_1 \partial_{t_1}t_2\nabla_{x_2} B \cdot  \partial_{t_1}  \Big(  \nabla_{x_2} F \cdot \widetilde{R^{(2)}_j \ast_2 R^{(1)}_i \ast  g} \Big)
     dx_1dx_2dt_1dt_2\\
    &\quad+ \int_{\mathbb R^{n+1}_+\times\mathbb R^{m+1}_+}  t_1 \partial_{t_1}t_2\nabla_{x_2} B \cdot  \partial_{t_1}  \Big(  \widetilde{R^{(2)}_j \ast_2 f} \cdot  \nabla_{x_2}  \widetilde {R^{(1)}_i \ast  g}  \Big)
     dx_1dx_2dt_1dt_2\\
   &\quad+ 2\int_{\mathbb R^{n+1}_+\times\mathbb R^{m+1}_+}  t_1 \partial_{t_1}t_2\partial_{x_{2,j}} B \cdot  \partial_{t_1}   \partial_{t_2} \Big(  \widetilde{ R^{(1)}_i \ast f}\cdot G  \Big)
     dx_1dx_2dt_1dt_2\\
    &\quad- \int_{\mathbb R^{n+1}_+\times\mathbb R^{m+1}_+}  t_1 \partial_{t_1}t_2\nabla_{x_2} B \cdot  \partial_{t_1}  \Big(  \nabla_{x_2} \widetilde{R^{(2)}_j \ast_2R^{(1)}_i \ast f}  \cdot G \Big)
     dx_1dx_2dt_1dt_2\\
    &\quad- \int_{\mathbb R^{n+1}_+\times\mathbb R^{m+1}_+}  t_1 \partial_{t_1}t_2\nabla_{x_2} B \cdot  \partial_{t_1}  \Big(  \widetilde{ R^{(1)}_i \ast f}  \cdot  \nabla_{x_2}  \widetilde {R^{(2)}_j \ast_2   g}  \Big)
     dx_1dx_2dt_1dt_2\\
    &\quad+ \int_{\mathbb R^{n+1}_+\times\mathbb R^{m+1}_+}  t_1 \partial_{t_1}t_2\nabla_{x_2} B \cdot  \partial_{t_1}  \Big(  \nabla_{x_2} \widetilde{ R^{(1)}_i \ast f} \cdot \widetilde{R^{(2)}_j \ast_2 g}  \Big)
     dx_1dx_2dt_1dt_2\\
    &\quad+ \int_{\mathbb R^{n+1}_+\times\mathbb R^{m+1}_+}  t_1 \partial_{t_1}t_2\nabla_{x_2} B \cdot  \partial_{t_1}  \Big(  \widetilde{R^{(2)}_j \ast_2  f} \cdot  \nabla_{x_2} G  \Big)
     dx_1dx_2dt_1dt_2\\
 &=: \mathcal {C}_{5,1}+\cdots+\mathcal {C}_{5,10}.
\end{align*}
By applying  Proposition 
\ref{prop BMO estimate} to these terms, we obtain that  they are all controlled by 
$$C\|b\|_{{\rm BMO}_{\mathcal{F}}(\mathbb R^n\times \mathbb R^m)} \|g\|_{L^p(\mathbb R^n\times \mathbb R^m)}
\|f\|_{L^{p'}(\mathbb R^n\times \mathbb R^m)}.$$
The estimates for the term $\mathcal C_6$ can be handled symmetrically, and we get the same upper bound for $\mathcal C_6$ as that for $\mathcal C_5$  above.

For the term $\mathcal C_7$, first note that 
$  \partial_{t_2}^2 B = -\Delta_{x_2} B = -\nabla_{x_2}\cdot \nabla_{x_2} B $. Hence we can write
\begin{align*}
\mathcal {C}_7=-&\int_{\mathbb R^{n+1}_+\times\mathbb R^{m+1}_+}
  t_1 t_2  \nabla_{x_2}  B \cdot   \nabla_{x_2}  \partial_{t_1}^2\Big( F  \cdot \widetilde{R^{(1)}_i \ast R^{(2)}_j\ast_2 g} + \widetilde{ R^{(1)}_i \ast f}\, \cdot \widetilde{R^{(2)}_j\ast_2 g} \\
  &\hskip1cm+ \widetilde{R^{(2)}_j \ast_2 f}\, \cdot \widetilde{R^{(1)}_i \ast g} + \widetilde{ R^{(2)}_j\ast_2 R^{(1)}_i\ast f}\, \cdot G\Big) dx_1dx_2dt_1dt_2.
  \end{align*}
Similar to the calculation in the terms $E_1$ and $E_2$ in the estimate of $\mathcal C_5$, we can now decompose 
\begin{align*}
\partial_{t_1}^2\Big( F  \cdot \widetilde{R^{(1)}_i \ast R^{(2)}_j\ast_2 g} + \widetilde{ R^{(1)}_i \ast f}\, \cdot \widetilde{R^{(2)}_j\ast_2 g} + \widetilde{R^{(2)}_j \ast_2 f}\, \cdot \widetilde{R^{(1)}_i \ast g} + \widetilde{ R^{(2)}_j\ast_2 R^{(1)}_i\ast f}\, \cdot G\Big) 
  \end{align*}
  into 10 terms, which further give 
  $$C_7= \mathcal {C}_{7,1}+\cdots+\mathcal {C}_{7,10}.$$
  Then by applying  Proposition 
\ref{prop BMO estimate} to these terms, we obtain that  they are all controlled by 
$$C\|b\|_{{\rm BMO}_{\mathcal{F}}(\mathbb R^n\times \mathbb R^m)} \|g\|_{L^p(\mathbb R^n\times \mathbb R^m)}
\|f\|_{L^{p'}(\mathbb R^n\times \mathbb R^m)}.$$
  
 The estimates for the term $\mathcal C_8$ can be handled symmetrically, and we get the same upper bound for $\mathcal C_7$ above. 
\end{proof}

\section{Upper bound of the big commutator $[b,R_{j,k}]$}\label{s4}

We derive a general upper bound result for commutators of any flag singular integral. The proof is based on the $A_{\mathcal{F},p}$ weighted estimate of flag singular integral operators and a Cauchy integral trick that goes back to the work of Coifman, Rochberg, and Weiss \cite{crw}. Roughly speaking, this technique allows one to bootstrap the weighted estimate for an arbitrary linear operator to that of its commutators of any order. This is the first time this idea is explored in the multi-parameter flag setting. In fact, although not needed for our upper bound proof, we demonstrate the bootstrapping result in the general higher order, two-weight setting. 

\subsection{$A_p$ weight and little bmo in the flag setting}\label{Ap}


To begin with, we define the Muckenhoupt $A_p$ weights in the flag setting, which consists of positive, locally integrable functions $w$ satisfying
\begin{equation}\label{Apdef}
[w]_{A_{\mathcal{F},p}}:=\sup_{R\in\mathcal{R}_{\mathcal{F}}}\left({1\over |R|}\int_R w(x,y) \,dx dy\right)\left({1\over |R|}\int_R w(x,y)^{1-p'}\,d xd y\right)^{p-1}<\infty,\quad 1<p<\infty,
\end{equation}where $p'$ denotes the H\"older conjugate of $p$. The following result of Wu \cite{W} provides a way of approaching the $A_{\mathcal{F},p}$ weights via the classical weights:
\begin{equation}\label{Apintersect}
A_{\mathcal{F},p}= A_p\cap A_p^{(2)},\qquad \forall 1<p<\infty,
\end{equation}where $A_p$ is the classical Muckenhoupt $A_p$ class of weights on $\R^{n+m}$, and $A_p^{(2)}$ consists of weights $w(x,y)$ such that $w(x,\cdot)\in A_p$ with uniformly bounded characteristics for a.e. fixed $x\in\R^n$.

We first show that a similar relation holds true for $\text{bmo}_{\mathcal{F}}$, which will be a useful tool for us in the study of this space.
\begin{lem}\label{bmointersect}
Let ${\rm BMO}(\R^{n+m})$ denote the classical John-Nirenberg BMO space on $\R^{n+m}$, and ${\rm BMO}^{(2)}(\mathbb{R}^m)$ be the space consisting of functions $f(x,y)$ such that $f(x,\cdot)\in\text{BMO}(\mathbb{R}^m)$ for a.e. fixed $x\in\mathbb{R}^n$ with uniformly bounded norm. There holds
\[
{\rm bmo}_{\mathcal{F}}(\mathbb{R}^{n+m})= {\rm BMO}(\R^{n+m})\cap {\rm BMO}^{(2)}(\mathbb{R}^m)
\]with comparable norms.
\end{lem}
\begin{proof}
The inclusion
\[
\text{bmo}_{\mathcal{F}}(\mathbb{R}^{n+m})\subset \text{BMO}(\R^{n+m})\cap \text{BMO}^{(2)}(\mathbb{R}^m)
\]can be easily verified. Indeed, the inclusion $\text{bmo}_{\mathcal{F}}(\mathbb{R}^{n+m})\subset \text{BMO}(\R^{n+m})$ is obvious from the definition. Now fix $x\in\mathbb{R}^n$. For any cube $J\subset\R^m$, one can find a sequence of cubes $I_k\subset\R^n$ such that $\ell(I_k)\leq\ell(J)$ and $I_k$ shrinks to the point $\{x\}$ as $k\to\infty$. The containment thus follows from the Lebesgue differentiation theorem. 

The other inclusion (``$\supset$'') of the lemma follows from Proposition \ref{explog} below, which establishes the exp-log connection between $A_{\mathcal{F},p}$ weights and $\text{bmo}_{\mathcal{F}}(\R^{n+m})$, similarly as in the one-parameter and the product setting.
\end{proof}

\begin{prop}\label{explog}
Suppose $w$ is a weight and $1<p<\infty$. We have
\begin{enumerate}
\item[{\rm(i)}] if $w\in A_{\mathcal{F},p}$, then $\log w\in \text{bmo}_{\mathcal{F}}(\R^{n+m})$;
\item[{\rm(ii)}] if $\log w\in \text{bmo}_{\mathcal{F}}(\R^{n+m})$, then $w^\eta \in A_{\mathcal{F},p}$ for sufficiently small $\eta>0$.
\end{enumerate}
\end{prop}
\begin{proof}
One observes directly from the definition that
\[
A_{\mathcal{F},p}\subset A_{\mathcal{F},q},\qquad \forall 1<p\leq q<\infty,
\]and
\[
w\in A_{\mathcal{F},p}\iff w^{1-p'}\in A_{\mathcal{F},p'},\qquad \forall 1<p<\infty.
\]Therefore, it suffices to prove the case $p=2$. 

We first prove (i). Suppose $w\in A_{\mathcal{F},2}$ and let $\varphi=\log w$. Then, for any $R\in\mathcal{R}_{\mathcal{F}}$ the $A_{\mathcal{F},2}$ condition implies that
\[
\left({1\over |R|}\int_R e^{\varphi(x,y)-\avg{\varphi}_R}\,dxdy\right)\left({1\over |R|}\int_R e^{\avg{\varphi}_R-\varphi(x,y)}\,dxdy\right)\leq[w]_{A_{\mathcal{F},2}}<\infty.
\]By Jensen's inequality we have each of the factors above is at least $1$ and at most $[w]_{A_{\mathcal{F},2}}$. Therefore, the inequality below holds:
\[
{1\over |R|}\int_R e^{|\varphi(x,y)-\avg{\varphi}_R|}\,dxdy\leq 2[w]_{A_{\mathcal{F},2}},
\]which, using the trivial estimate $t\leq e^t$, implies that
\[
{1\over |R|}\int_R |\varphi(x,y)-\avg{\varphi}_R|\,dxdy\leq 2[w]_{A_{\mathcal{F},2}}.
\]Hence, $\varphi\in\text{bmo}_{\mathcal{F}}(\mathbb{R}^{n+m})$.

We now prove (ii). Let $\varphi=\log w\in\text{bmo}_{\mathcal{F}}(\R^{n+m})$, it follows from (\ref{bmointersect}) that $\varphi\in \text{BMO}(\R^{n+m})$ and $\varphi\in\text{BMO}^{(2)}(\mathbb{R}^m)$. According to the classical exp-log connection between BMO and $A_2$, there hold for sufficiently small $\eta>0$ that
\[
e^{\eta\varphi(\cdot,\cdot)}\in A_2(\R^{n+m})
\]and
\[
e^{\eta\varphi(x,\cdot)}\in A_2(\R^m)\qquad \text{uniformly in}\,\, x\in\R^n.
\]Hence, (\ref{Apintersect}) implies that $e^{\eta\varphi}\in A_{\mathcal{F},2}$ for sufficiently small $\eta>0$, which completes the proof.
\end{proof}

\subsection{Upper bound of the commutator}

Given an operator $T$, define its $k$-th order commutator as
\[
C^{k}_{\vec{b}}(T):=[b_k,[b_{k-1},\cdots,[b_1,T]\cdots]],
\]where each $b_j$ is a function on $\mathbb{R}^n\times\mathbb{R}^m$, $\forall 1\leq j\leq k$.
\begin{thm}\label{bootstrap}
Let $\nu$ be a fixed weight on $\mathbb{R}^n\times\mathbb{R}^m$, $1<p<\infty$, and $T$ be a linear operator satisfying
\[
\|T\|_{L^p(\mu)\to L^p(\lambda)}\leq C_{n,m,p,T}\left([\mu]_{A_{\mathcal{F},p}},[\lambda]_{A_{\mathcal{F},p}}\right),
\]where $C_{n,m,p,T}(\cdot,\cdot)$ is an increasing function of both components, with $\mu, \lambda\in A_{\mathcal{F},p}$ and $\mu/\lambda=\nu^p$. For $k\geq 1$, let $b_j\in \text{bmo}_{\mathcal{F}}(\mathbb{R}^n\times\mathbb{R}^m)$, $1\leq j\leq k$, then there holds
\[
\|C^k_{\vec{b}}(T)\|_{L^p(\mu)\to L^p(\lambda)}\leq C_{n,m,p,k,T}\left([\mu]_{A_{\mathcal{F},p}},[\lambda]_{A_{\mathcal{F},p}}\right)\prod_{j=1}^k\|b_j\|_{\text{bmo}_{\mathcal{F}}}.
\]
\end{thm}

Assuming Theorem \ref{bootstrap}, in order to derive (even unweighted) upper estimate for commutator of operator $T$, it suffices to know the corresponding weighted estimate for $T$ itself. When $T$ is a flag singular integral operator (which includes the flag Riesz transform $R_{j,k}$), such a result is obtained by Han, Lin and Wu in \cite{HLW}.
\begin{defn}\label{FlagSIO}
A flag singular integral $T_\mathcal{F}:\,f\mapsto \mathcal{K}*f$ is defined via a flag kernel $\mathcal{K}$ on $\R^n\times \R^m$, which is a distribution on $\R^{n+m}$ that coincides with a $C^\infty$ function away from the coordinate subspace $\{ (0,y) \}\subset\R^{n+m}$ and satisfies
\begin{enumerate}
  \item[\rm (i)] (differential inequalities) For each $\alpha=(\alpha_1, \ldots \alpha_n)$, $\beta=(\beta_1, \ldots \beta_n)$
$$ |\partial_{x}^{\alpha}\partial_{y}^{\beta} \mathcal{K}(x,y)|\lesssim |x|^{-n-|\alpha|}
 (|x|+|y|)^{-m-|\beta|}
$$
for all $(x,y)\in \R^{n+m}$ with $|x|\not= 0;$
  \item[\rm (ii)] (cancellation conditions)
  \begin{enumerate}
    \item $$\bigg| \int_{\R^m} \partial_{x}^{\alpha}\mathcal{K}(x,y) \psi_1(\delta y)dy \bigg|\leq C_\alpha |x|^{-n-|\alpha|}$$
    for every multi-index $\alpha$ and for every normalized bump function $\psi_1$ on $\R^m$ and every $\delta>0$;
  \item $$\bigg| \int_{\R^n} \partial_{y}^{\beta}\mathcal{K}(x,y) \psi_2(\delta y)dy \bigg|\leq C_\beta |y|^{-m-|\beta|}$$
    for every multi-index $\beta$ and for every normalized bump function $\psi_2$ on $\R^n$ and every $\delta>0$;
  \item $$\bigg| \int_{\R^{n+m}} \mathcal{K}(x,y) \psi_3(\delta_1x,\delta_2 y)dxdy \bigg|\leq C$$
   for every normalized bump function $\psi_3$ on $\R^{n+m}$ and every $\delta_1,\delta_2>0$.
     \end{enumerate}
   \end{enumerate}
\end{defn}

\begin{thm}[Remark 1.4 of \cite{HLW}]\label{borrow}
Let $1<p<\infty$ and $w\in A_{\mathcal{F},p}(\R^{n+m})$, there holds
\[
\|T_\mathcal{F}(f)\|_{L^p_w(\R^{n+m})}\leq C_p \|f\|_{L^p_w(\R^{n+m})},\quad \forall f\in L^p_w(\R^{n+m}).
\]
\end{thm}

Applying Theorem \ref{bootstrap} (with the choice $\mu=\lambda=w$) together with Theorem \ref{borrow}, one obtains immediately the following.
\begin{cor}
Let $w\in A_{\mathcal{F},p}$, $1<p<\infty$ and $T$ be a flag singular integral operator as defined above. For any $k\geq 1$, $\vec{b}=(b_1,\cdots,b_k)$ where $b_j\in\text{bmo}_{\mathcal{F}}(\mathbb{R}^n\times\mathbb{R}^m)$, $j=1,\ldots,k$, there holds
\[
\|C^{k}_{\vec{b}}(T)\|_{L^p(w)\to L^p(w)}\leq C_{n,m,p,k,w,T} \prod_{j=1}^k\|b_j\|_{\text{bmo}_{\mathcal{F}}}.
\]

\end{cor}
Obviously, the result above in the first order unweighted case is precisely the desired upper bound estimate in Theorem \ref{thm big com}.

The core of the proof of Theorem \ref{bootstrap} lies in a complex function representation of the commutators and the Cauchy integral formula. This method has been widely used to obtain upper estimates for linear and multilinear commutators in various settings, see \cite{CPP, crw, Hy, BMMST, KO} for examples. The main new challenge in our problem is the unique structure of the little flag BMO space and flag weights, which for instance doesn't seem to fall into the category of spaces recently studied in \cite{BMMST}. 
\begin{proof}[Proof of Theorem \ref{bootstrap}]
Observe that
\[
C^k_{\vec{b}}(T)=\partial_{z_1}\cdots\partial_{z_k}F(\vec{0}),\qquad F(\vec{z}):=e^{\sum_{j=1}^k b_1z_1}Te^{-\sum_{j=1}^k b_jz_j},
\]which generalizes a classical formula representing higher order commutators. We remark that when all the symbol functions $b_j$ are the same, one can work instead with a simpler formula using single variable complex functions and their $k$-th order derivatives. 
According to the Cauchy integral formula on polydiscs,
\[
C_{\vec{b}}^k(T)=\frac{1}{(2\pi i)^k}\oint \cdots\oint \frac{F(\vec{z})\,d z_1\cdots d z_k}{z_1^2\cdots z_k^2}
\]where each integral is over any closed path around the origin in the corresponding variable. For fixed $(\delta_1,\ldots,\delta_k)$ which will be determined later, there holds by Minkowski inequality that
\[
\begin{split}
&\|C^k_{\vec{b}}(T)\|_{L^p(\mu)\to L^p(\lambda)}\\
\leq&\frac{1}{(2\pi)^k}\oint _{|z_1|=\delta_1}\cdots\oint_{|z_k|=\delta_k}\|T\|_{L^p\left(e^{p\operatorname{Re}(\sum_{j=1}^kb_jz_j)}\mu\right)\to L^p\left(e^{p\operatorname{Re}(\sum_{j=1}^kb_jz_j)}\lambda\right)}\frac{|\d z_1|\cdots |\d z_k|}{\delta_1^2\cdots \delta_k^2}\\
\leq&\frac{1}{(2\pi)^k}\oint _{|z_1|=\delta_1}\cdots\oint_{|z_k|=\delta_k}C_{n,m,p,T}\left([e^{p\operatorname{Re}(\sum_{j=1}^kb_jz_j)}\mu]_{A_{\mathcal{F},p}},[e^{p\operatorname{Re}(\sum_{j=1}^kb_jz_j)}\lambda]_{A_{\mathcal{F},p}}\right)\frac{|\d z_1|\cdots |\d z_k|}{\delta_1^2\cdots\delta_k^2},
\end{split}
\]where we have used the fact that $(e^{p\operatorname{Re}(\sum_{j=1}^kb_jz_j)}\mu, e^{p\operatorname{Re}(\sum_{j=1}^kb_jz_j)}\lambda)$ is a pair of weights satisfying
\[
\frac{e^{p\operatorname{Re}(\sum_{j=1}^kb_jz_j)}\mu}{e^{p\operatorname{Re}(\sum_{j=1}^kb_jz_j)}\lambda}=\frac{\mu}{\lambda}=\nu^p.
\]

Now we choose $\{\delta_j\}$ according to Lemma \ref{lem} below, which is the key ingredient of the proof concerning the relation between $A_{\mathcal{F},p}$ weights and little flag BMO functions. Let
\[
\delta_1:=\frac{\epsilon_{n,m,p}}{\max\left((\mu)_{A_{\mathcal{F},p}},(\lambda)_{A_{\mathcal{F},p}}\right)\|b_1\|_{\text{bmo}_{\mathcal{F}}}},
\]where for any $w\in A_{\mathcal{F},p}$
\begin{align}\label{(w)}
(w)_{A_{\mathcal{F},p}}:=\max\left([w]_{A_{\mathcal{F},p}}, [\sigma]_{A_{\mathcal{F},p'}}\right).
\end{align}
Here we have used the notation $\sigma:=w^{1-p'}$ to denote the dual weight of $w$, and the relevant property of $(w)_{A_{\mathcal{F},p}}$ to us is that
\[
(w)_{A_{\mathcal{F},p}}=\max([w]_{A_{\mathcal{F},p}},[w]_{A_{\mathcal{F},p}}^{p'-1})=[w]_{A_{\mathcal{F},p}}^{\max(1,p'-1)}.
\]Recursively, for any $j\geq 2$, choose
\[
\delta_j:=\frac{\epsilon_{n,m,p}}{\sup_{\{z_t\}:\,|z_1|=\delta_1,\ldots,|z_{j-1}|=\delta_{j-1}}\max\left(\big(e^{p\operatorname{Re}(\sum_{t=1}^{j-1}b_tz_t)}\mu\big)_{A_{\mathcal{F},p}},\big(e^{p\operatorname{Re}(\sum_{t=1}^{j-1}b_tz_t)}\lambda\big)_{A_{\mathcal{F},p}}\right)\|b_j\|_{\text{bmo}_{\mathcal{F}}}}.
\]
Then applying Lemma \ref{lem} iteratively shows that
\[
[e^{p\operatorname{Re}(\sum_{j=1}^kb_jz_j)}\mu]_{A_{\mathcal{F},p}}\leq C_{n,m,p}[e^{p\operatorname{Re}(\sum_{j=1}^{k-1}b_jz_j)}\mu]_{A_{\mathcal{F},p}}\leq \cdots\leq C_{n,m,p}^k[\mu]_{A_{\mathcal{F},p}},
\]and similarly
\[
[e^{p\operatorname{Re}(\sum_{j=1}^kb_jz_j)}\lambda]_{A_{\mathcal{F},p}}\leq C_{n,m,p}^k[\lambda]_{A_{\mathcal{F},p}},
\]which in turn via the monotonicity of $C_{n,m,p,T}(\cdot,\cdot)$ leads to
\[
C_{n,m,p,T}\left([e^{p\operatorname{Re}(\sum_{j=1}^kb_jz_j)}\mu]_{A_{\mathcal{F},p}},[e^{p\operatorname{Re}(\sum_{j=1}^kb_jz_j)}\lambda]_{A_{\mathcal{F},p}}\right)\leq C'_{n,m,p,k,T}\left([\mu]_{A_{\mathcal{F},p}},[\lambda]_{A_{\mathcal{F},p}}\right).
\]Therefore,
\[
\begin{split}
\|C^k_{\vec{b}}(T)\|_{L^p(\mu)\to L^p(\lambda)}\leq& \frac{1}{\delta_1\cdots\delta_k}C'_{n,m,p,k,T}\left([\mu]_{A_{\mathcal{F},p}},[\lambda]_{A_{\mathcal{F},p}}\right)\\
\leq& C_{n,m,p,k,T}\left([\mu]_{A_{\mathcal{F},p}},[\lambda]_{A_{\mathcal{F},p}}\right)\prod_{j=1}^k\|b_j\|_{\text{bmo}_{\mathcal{F}}}.
\end{split}
\]The proof is thus complete.
\end{proof}
\begin{lem}\label{lem}
Let $w\in A_{\mathcal{F},p}$, $1<p<\infty$, and $b\in\text{bmo}_{\mathcal{F}}(\mathbb{R}^n\times\mathbb{R}^m)$. There are constants $\epsilon_{n,m,p}, C_{n,m,p}>0$ such that
\[
[e^{\operatorname{Re}(bz)}w]_{A_{\mathcal{F},p}}\leq C_{n,m,p}[w]_{A_{\mathcal{F},p}}
\]whenever $z\in\mathbb{C}$ satisfies
\[
|z|\leq\frac{\epsilon_{n,m,p}}{\|b\|_{\text{bmo}_{\mathcal{F}}}(w)_{A_{\mathcal{F},p}}},
\]
where $(w)_{A_{\mathcal{F},p}}$ is defined as in \eqref{(w)}.     
\end{lem}

\begin{proof}
This estimate is a consequence of (\ref{Apintersect}), Lemma \ref{bmointersect} and a one-parameter version proven by Hyt\"onen in \cite{Hy}, which states that for any $w\in A_p$, the classical Muckenhoupt $A_p$ class on $\mathbb{R}^d$, $1<p<\infty$, there exist $\epsilon_{d,p}, C_{d,p}>0$ such that 
\[
[e^{\operatorname{Re}(bz)}w]_{A_p}\leq C_{d,p}[w]_{A_p}
\]for all $z\in\mathbb{C}$ with  
\[
|z|\leq\frac{\epsilon_{n,p}}{\|b\|_{\text{BMO}}(w)_{A_p}}.
\]

To see this, by (\ref{Apintersect}) and Lemma \ref{bmointersect}, given $w\in A_{\mathcal{F},p}$ and $b\in \text{bmo}_{\mathcal{F}}$, there hold $w\in A_p\cap A_p^{(2)}$ and $b\in \text{BMO}(\R^{n+m})\cap \text{BMO}^{(2)}(\R^m)$. Hence, taking $\epsilon_{n,m,p}>0$ sufficiently small, for all $z\in\mathbb{C}$ satisfying
\[
|z|\leq\frac{\epsilon_{n,m,p}}{\|b\|_{\text{bmo}_{\mathcal{F}}}(w)_{A_{\mathcal{F},p}}},
\]one has
\[
[e^{\operatorname{Re}(bz)}w]_{A_p}\leq C_{n+m,p}[w]_{A_p}\leq C_{n,m,p}[w]_{A_{\mathcal{F},p}}
\]and
\[
[e^{\operatorname{Re}(b(x,\cdot)z)}w(x,\cdot)]_{A_p}\leq C_{m,p}[w(x,\cdot)]_{A_p}\leq C_{n,m,p}[w]_{A_{\mathcal{F},p}},\qquad \text{a.e.}\, x\in \R^n,
\]by observing that $$\|b\|_{\text{bmo}_{\mathcal{F}}}\gtrsim \max\big(\|b\|_{\text{BMO}(\R^{n+m})},\sup_{x\in\R^n}\|b(x,\cdot)\|_{\text{BMO}^{(2)}(\R^m)}\big)$$ and that 
$$(w)_{A_{\mathcal{F},p}}\gtrsim \max([w]_{A_p},\sup_{x\in\R^n}[w(x,\cdot)]_{A_p}).$$ The proof is thus complete.
\end{proof}

\section{Applications: div-curl lemmas in the flag setting}\label{s5}

Let $E^{(1)}$ be a vector field on $\R^{n+m}$ taking the values in $\R^{n+m}$, and let
$E^{(2)}$ be a vector field on $\R^{m}$ taking the values in $\R^{m}$. 
Now let $\mathcal M_{{n+m},m}$ denote the set of all $(n+m) \times m$ matrices.
We now consider the following
version of vector fields on $\R^{n}\times\R^m$ taking the values in $\mathcal M_{{n+m},m}$, associated  with the flag structure:
\begin{align}\label{flag vector field}
E=E^{(1)}\ast_2 E^{(2)}:=
\begin{bmatrix}
    E^{(1)}_{1}\ast_2 E^{(2)}_1        & \dots & E^{(1)}_{1}\ast_2 E^{(2)}_m \\
        \vdots    & \dots &  \vdots \\
        \vdots    & \dots &  \vdots \\
         \vdots    & \dots &  \vdots \\
   E^{(1)}_{n+m}\ast_2 E^{(2)}_1        & \dots & E^{(1)}_{n+m}\ast_2 E^{(2)}_m
\end{bmatrix},
\end{align}
where $$E^{(1)}_{j}\ast_2 E^{(2)}_k(x,y) =\int_{\R^m} E^{(1)}_{j}(x,y-z) E^{(2)}_k(z)\,dz.  $$

Next we consider the following $L^p$ space via projections. Suppose $1<p<\infty$. We define
$L^p_{\mathcal{F}}(\R^n\times\R^m; \mathcal M_{{n+m},m})$ to be the set of vector fields $E$ 
in $L^p(\R^n\times\R^m; \mathcal M_{{n+m},m})$ such that there exist $r_1,r_2\in(1,\infty)$ with ${1\over r_1}+{1\over r_2}={1\over p}+1$, 
$E^{(1)}\in L^{r_1}(\R^{n+m};\R^{n+m})$, $E^{(2)}\in L^{r_2}(\R^{m};\R^{m})$ and that
$E=E^{(1)}\ast_2 E^{(2)}$, moreover,
$$ \|E\|_{L^p_{\mathcal{F}}(\R^n\times\R^m; \mathcal M_{{n+m},m})} :=\inf\|E^{(1)}\|_{L^{r_1}(\R^{n+m};\R^{n+m})}\|E^{(2)}\|_{L^{r_2}(\R^{m};\R^{m})},$$
where the infimum is taken over all possible $r_1,r_2\in(1,\infty),E^{(1)}\in L^{r_1}(\R^{n+m};\R^{n+m}), E^{(2)}\in L^{r_2}(\R^{m};\R^{m})$.

Given two matrices $A, B \in \mathcal M_{{n+m},m}$, we define the ``dot product'' between $A$ and $B$ by
$$ A\cdot B = \sum_{j=1}^{n+m}\sum_{k=1}^m A_{j,k}B_{j,k}. $$
We point out that this is the Hilbert-Schmidt inner product for two matrices and more generally this is referred to as the Schur product of two matrices.

We now prove Theorem \ref{thm divcurl1}.
\begin{proof}[Proof of Theorem \ref{thm divcurl1}]
Note that $B$  is a vector field on $\R^{n}\times\R^m$ taking the values in $\mathcal M_{{n+m},m}$, associated  with the flag structure \eqref{flag vector field}. Then there exist certain vector field $B^{(1)}$  on $\R^{n+m}$ taking the values in $\R^{n+m}$ and   vector field $B^{(2)}$  on $\R^{m}$ taking the values in $\R^{m}$ such that
$B=B^{(1)}\ast_2 B^{(2)}$ and that 
$$ \|B\|_{L^q_{\mathcal{F}}(\R^n\times\R^m; \mathcal M_{{n+m},m})} \approx \inf\|B^{(1)}\|_{L^{q_1}(\R^{n+m};\R^{n+m})}\|B^{(2)}\|_{L^{q_2}(\R^{m};\R^{m})}$$
with ${1\over q_1}+{1\over q_2}={1\over q}+1$.

Thus, $\operatorname{curl}_{(x,y)} B^{(1)}=0$ implies that there exists $\phi^{(1)}\in L^q(\R^{n+m})$ such that
$$ B^{(1)} = ( R^{(1)}_1\phi^{(1)},\ldots,R^{(1)}_{n+m}\phi^{(1)}  )$$
with $\|B^{(1)}\|_{L^{q_1}(\R^{n+m};\R^{n+m})}\approx \|\phi^{(1)}\|_{L^{q_1}(\R^{n+m})}$.
Again, $\operatorname{curl}_{y} B^{(2)}=0$ implies that there exists $\phi^{(2)}\in L^{q_2}(\R^{n+m})$ such that
$$ B^{(2)} = ( R^{(2)}_1\phi^{(2)},\ldots,R^{(2)}_{m}\phi^{(2)}  )$$
with $\|B^{(2)}\|_{L^{q_2}(\R^{m};\R^{m})}\approx \|\phi^{(2)}\|_{L^{q_2}(\R^{m})}$. 
As a consequence we get that the matrix $B$ has elements
$$B_{j,k}= R_{j,k} \ast \phi,\quad j=1,\ldots,n+m,\ k=1,\ldots,m, $$
where $\phi = \phi^{(1)}\ast_2 \phi^{(2)}$ and 
$\|B\|_{L^q_{\mathcal{F}}(\R^n\times\R^m; \mathcal M_{{n+m},m})} \approx \|\phi\|_{L^q(\R^{n+m})}$.

Similarly, note that $E$  is a vector field on $\R^{n}\times\R^m$ taking the values in $\mathcal M_{{n+m},m}$, associated  with the flag structure \eqref{flag vector field}. Then there exist certain vector field $E^{(1)}$  on $\R^{n+m}$ taking the values in $\R^{n+m}$ and   vector field $E^{(2)}$  on $\R^{m}$ taking the values in $\R^{m}$ such that
$E=E^{(1)}\ast_2 E^{(2)}$ and that 
$$ \|E\|_{L^p_{\mathcal{F}}(\R^n\times\R^m; \mathcal M_{{n+m},m})} \approx \inf\|E^{(1)}\|_{L^{p_1}(\R^{n+m};\R^{n+m})}\|E^{(2)}\|_{L^{p_2}(\R^{m};\R^{m})}$$
with ${1\over p_1}+{1\over p_2}={1\over p}+1$.

Thus, the conditions  $\operatorname{div}_{(x,y)} E^{(1)}=0$ and 
 $\operatorname{div}_{y} E^{(2)}=0$ imply that
$$\sum_{j=1}^{n+m} R^{(1)}_j\ast E^{(1)}_j(x,y)=0\quad{\rm and}\quad  \sum_{k=1}^{m} R^{(2)}_k\ast_2 E^{(2)}_k(y)=0. $$
Hence, we get that
$$ \sum_{j=1}^{n+m}  R^{(1)}_j\ast E_{j,k}(x,y)=0\quad{\rm and}\quad   \sum_{k=1}^{m} R^{(2)}_k\ast_2E_{j,k}(x,y)=0. $$
With these facts, we have that
\begin{align*}
E(x,y)\cdot B(x,y)& =  \sum_{j=1}^{n+m}  \sum_{k=1}^{m} E_{j,k}(x,y)B_{j,k}(x,y)= \sum_{j=1}^{n+m}  \sum_{k=1}^{m} E_{j,k}(x,y)R_{j,k} \ast \phi(x,y)\\
&=\sum_{j=1}^{n+m}  \sum_{k=1}^{m}\bigg\{E_{j,k}(x,y)R_{j,k} \ast \phi(x,y) + R^{(1)}_j\ast E_{j,k}(x,y) R^{(2)}_k \ast_2 \phi(x,y)\\
&\hskip2cm + R^{(2)}_k\ast_2 E_{j,k}(x,y) R^{(1)}_j\ast \phi(x,y)  + R_{j,k} \ast E_{j,k}(x,y) \phi(x,y)  \bigg\}
\end{align*}
Now testing this equality over all functions in the flag BMO space, i.e., for every $b\in {\rm BMO}_{\mathcal{F}}(\R^n\times\R^m)$, 
and then unravelling the expression with the Riesz transforms we see that
\begin{align*}
&\int_{\R^n\times \R^m} E(x,y)\cdot B(x,y)\, b(x,y)\, dxdy\\
&=\sum_{j=1}^{n+m}  \sum_{k=1}^{m}\int_{\R^n\times \R^m}  \big[[b,R^{(1)}_j], R^{(2)}_k\big]_2(E_{j,k})(x,y) \phi(x,y)\, dxdy.
\end{align*}
Then based on Theorem \ref{thm iterated com}, since $b\in {\rm BMO}_{\mathcal{F}}(\R^n\times\R^m)$ we have that each of the above commutators is a bounded operator on $L^p(\R^n \times \R^m)$ with norm controlled by the norm of $b$, i.e.,
$\|b\|_{ {\rm BMO}_{\mathcal{F}}(\R^n\times\R^m)}.$

As a consequence, we get that 
\begin{align*}
&\bigg|\int_{\R^n\times \R^m} E(x,y)\cdot B(x,y)\, b(x,y)\, dxdy\bigg|\\
&\ls  \|b\|_{ {\rm BMO}_{\mathcal{F}}(\R^n\times\R^m)}\|E\|_{L^p_{\mathcal{F}}(\R^n\times\R^m; \mathcal M_{{n+m},m})} \|\phi\|_{L^q(\R^{n+m})}\\
&\ls  \|b\|_{ {\rm BMO}_{\mathcal{F}}(\R^n\times\R^m)}\|E\|_{L^p_{\mathcal{F}}(\R^n\times\R^m; \mathcal M_{{n+m},m})} \|B\|_{L^q_{\mathcal{F}}(\R^n\times\R^m; \mathcal M_{{n+m},m})}.
\end{align*}
Then from the duality of $H^1_{\mathcal{F}}(\R^n\times\R^m)$ with ${\rm BMO}_{\mathcal{F}}(\R^n\times\R^m)$, we obtain that
\begin{align*}
\|E\cdot B\|_{H^1_{\mathcal{F}}(\R^n\times\R^m)}\ls  \|b\|_{ {\rm BMO}_{\mathcal{F}}(\R^n\times\R^m)}\|E\|_{L^p_{\mathcal{F}}(\R^n\times\R^m; \mathcal M_{{n+m},m})} \|B\|_{L^q_{\mathcal{F}}(\R^n\times\R^m; \mathcal M_{{n+m},m})}.
\end{align*}
This finishes the proof of Theorem \ref{thm divcurl1}.
\end{proof}

\begin{proof}[Proof of Theorem \ref{thm divcurl2}]
Suppose that 
$E,B$ are vector fields on $\R^{n}\times\R^m$ taking values in $\R^{n+m}$. Moreover, suppose
$E\in L^p(\R^{n}\times\R^m; \R^{n+m})$ and 
$B\in L^q(\R^{n}\times\R^m; \R^{n+m})$ satisfy that
$$ \operatorname{div}_{(x,y)} E(x,y)=0\quad{\rm and}\quad \operatorname{curl}_{(x,y)} B(x,y)=0 $$
and
$$ \operatorname{div}_{y} E(x,y)=0\quad{\rm and}\quad \operatorname{curl}_{y} B(x,y)=0  ,\quad \forall x\in\R^n.$$


We now define the projection operator $\mathcal P$ as
$$ \mathcal P E =\bigg( E_1+ R^{(1)}_1\Big( \sum_{k=1}^{n+m}R^{(1)}_kE_k \Big),\ldots, E_{n+m}+ R^{(1)}_{n+m}\Big( \sum_{k=1}^{n+m}R^{(1)}_kE_k \Big) \bigg). $$
Then by definition, it is direct that
$$ \operatorname{div}_{(x,y)} \mathcal P E  =0 $$
since
\begin{align}\label{ee1}
\sum_{j=1}^{n+m} R^{(1)}_j\bigg( E_j+ R^{(1)}_j\bigg(\sum_{k=1}^{n+m} R^{(1)}_k E_k \bigg) \bigg)=0.
\end{align}

Moreover, we also have $ \mathcal P \circ \mathcal P E =  \mathcal P E$. Next, we point out that
applying $[b,\mathcal P]$ to the vector field $E$, we can get that the $j$th component is given by
$$ \sum_{k=1}^{n+m} [ b,  R^{(1)}_jR^{(1)}_k](E_k). $$

Suppose now $b\in {\rm bmo}_{\mathcal{F}}(\R^n\times \R^m)$. Then from Lemma \ref{bmointersect} we know that
$$\text{bmo}_{\mathcal{F}}(\mathbb{R}^{n+m})= \text{BMO}(\R^{n+m})\cap \text{BMO}^{(2)}(\mathbb{R}^m)$$
with comparable norms. Hence, we have that $b\in  \text{BMO}(\R^{n+m})$ with 
$$\|b\|_{ \text{BMO}(\R^{n+m})} \ls \|b\|_{ {\rm bmo}_{\mathcal{F}}(\R^n\times \R^m)}.$$
With all these observations,  an application of the Coifman, Rochberg and Weiss theorem demonstrates
that $[b,\mathcal P](E)$ is bounded on $L^p(\R^{n}\times\R^m; \R^{n+m})$ with 
\begin{align*} 
\|[b,\mathcal P](E)\|_{L^p(\R^{n}\times\R^m; \R^{n+m})}&\ls \|b\|_{ \text{BMO}(\R^{n+m})}\|E\|_{L^p(\R^{n}\times\R^m; \R^{n+m})}\\
&\ls \|b\|_{ {\rm bmo}_{\mathcal{F}}(\R^n\times \R^m)}\|E\|_{L^p(\R^{n}\times\R^m; \R^{n+m})}. 
\end{align*}
As a consequence, from the definition of $[b,P]$ and \eqref{ee1} we get that
\begin{align*}
\bigg| \int_{\R^{n+m}} E(x,y)\cdot B(x,y)\, b(x,y)\,dxdy \bigg| &= \bigg| \int_{\R^{n+m}} [b, P]E(x,y)\cdot B(x,y)\,dxdy \bigg|\\
&\ls \|b\|_{ \text{BMO}(\R^{n+m})}\|E\|_{L^p(\R^{n}\times\R^m; \R^{n+m})} \|B\|_{L^q(\R^{n}\times\R^m; \R^{n+m})}\\
&\ls \|b\|_{ {\rm bmo}_{\mathcal{F}}(\R^n\times \R^m)} \|E\|_{L^p(\R^{n}\times\R^m; \R^{n+m})}\|B\|_{L^q(\R^{n}\times\R^m; \R^{n+m})}.
\end{align*}
Thus we get that $E\cdot B$ is in $H^1(\R^{n+m})$ with 
$$ \|E\cdot B\|_{H^1(\R^{n+m})}\ls  \|E\|_{L^p(\R^{n}\times\R^m; \R^{n+m})}\|B\|_{L^q(\R^{n}\times\R^m; \R^{n+m})}.$$

To show the second result, we now define the projection operator $\mathcal P^{(2)}$ as
$$ \mathcal P^{(2)} E =\bigg( E_{n+1}+ R^{(2)}_1\Big( \sum_{k=1}^{m}R^{(2)}_kE_{n+k} \Big),\ldots, E_{n+m}+ R^{(1)}_{n+m}\Big( \sum_{k=1}^{m}R^{(1)}_kE_{n+k} \Big) \bigg). $$
Then, again, by definition, we have that
$$ \operatorname{div}_{y} \mathcal P^{(2)} E  =0 $$
since
\begin{align}\label{ee2}
\sum_{j=1}^{m} R^{(2)}_j\bigg( E_{n+j}+ R^{(2)}_j\bigg(\sum_{k=1}^{m} R^{(2)}_k E_{n+k} \bigg) \bigg)=0.
\end{align}

Now fix $x\in\R^n$, by using the definition of  $\mathcal P^{(2)}$ and the fact  \eqref{ee2}, we get that for 
$b\in {\rm bmo}_{\mathcal{F}}(\R^n\times \R^m)$,
\begin{align*}
\int_{\R^m} E(x,y)\cdot_2 B(x,y) b(x,y)\,dy = \int_{\R^m} [b(x,\cdot),\mathcal P^{(2)}]E(x,y)  \psi(x,y) \,dy.
\end{align*}
Integrating the above equality over $\R^n$, we have
\begin{align*}
&\hskip-.5cm\bigg|\int_{\R^n}\int_{\R^m} E(x,y)\cdot_2 B(x,y) b(x,y)\,dydx \bigg|\\
&=\bigg| \int_{\R^n}\int_{\R^m} [b(x,\cdot),\mathcal P^{(2)}]E(x,y)  \cdot_2 B(x,y)  \,dydx\bigg|\\
&\ls  \int_{\R^n} \|b(x,\cdot)\|_{{\rm BMO}(\R^n)} \|E(x,\cdot)\|_{L^p(\R^m)} \| B(x,\cdot)\|_{L^q(\R^m)} \,dx\\
&\ls \|b\|_{{\rm bmo}_{\mathcal{F}}(\R^n\times \R^m)} \int_{\R^n}  \|E(x,\cdot)\|_{L^p(\R^m)} \| B(x,\cdot)\|_{L^q(\R^m)} \,dx\\
&\ls \|b\|_{{\rm bmo}_{\mathcal{F}}(\R^n\times \R^m)}  \|E\|_{L^p(\R^m\times\R^n;\R^{n+m})} \| B\|_{L^p(\R^m\times\R^n;\R^{n+m})}. 
\end{align*}
Here we use again Lemma \ref{bmointersect}  and H\"older's inequality.  Taking the supremum over all $b\in {\rm bmo}_{\mathcal{F}}(\R^n\times \R^m)$ we obtain that
\begin{align*}
\int_{\R^m}\|E(\cdot,y)\cdot_2 B(\cdot,y)\|_{H^1(\R^{m})}\,dy \ls \|E\|_{L^p(\R^{n}\times\R^m; \R^{n+m})}
\|B\|_{L^q(\R^{n}\times\R^m; \R^{n+m})}.
\end{align*}
This finishes the proof of Theorem \ref{thm divcurl2}.
\end{proof}

\medskip

{\bf Acknowledgments:}  
 X. T. Duong, J. Li and J. Pipher  are supported by ARC DP 160100153.  X. T. Duong and J. Li are also supported by
Macquarie University Research Seeding Grant. 
B. D. Wick's research is supported in part by National Science Foundation grant 
DMS \#1560955. This paper started in July 2016 during J. Li's visiting J. Pipher at Brown University. J. Li would like to thank the Department of Mathematics of Brown University for its hospitality. 
Part of this material is based upon work supported by the National Science Foundation under Grant No. DMS-1440140 while Y. Ou and J. Pipher were in residence at the Mathematical Sciences Research Institute in Berkeley, California, during the Spring 2017 semester.

{\small
\medskip


\smallskip

X. Duong, Department of Mathematics, Macquarie University, NSW, 2109, Australia.

\smallskip

{\it E-mail}: \texttt{xuan.duong@mq.edu.au}

\vspace{0.3cm}

%
%
J. Li, Department of Mathematics, Macquarie University, NSW, 2109, Australia.

\smallskip

{\it E-mail}: \texttt{ji.li@mq.edu.au}

\vspace{0.3cm}

Y. Ou, Department of Mathematics, Baruch College, CUNY, New York, NY 10025, USA

\smallskip

{\it E-mail}: \texttt{yumeng.ou@baruch.cuny.edu}

\vspace{0.3cm}
%
%
%
J. Pipher, Department of Mathematics, Brown University, Providence,  RI 02912, USA

\smallskip

{\it E-mail}: \texttt{jpipher@math.brown.edu}

\vspace{0.3cm}

B. Wick, Department of Mathematics, Washington University--St. Louis, St. Louis, MO 63130-4899 USA

\smallskip

{\it E-mail}: \texttt{wick@math.wustl.edu}

\vspace{0.3cm}
}

\end{document}